%% file: ExamplesCanonicalTopology.tex
\documentclass{article}
\usepackage{cite} 
\usepackage{amsthm} 
\usepackage{amsfonts} 
\usepackage{amsmath} 
\usepackage{amscd} 
\usepackage{tikz} 
\usepackage{dutchcal} 
\usepackage{mathrsfs} 
\usepackage{tikz-cd} 
\usepackage{mathtools} 
\usepackage{amssymb} 
\usepackage{euscript}
\usepackage{mathtools}

\usepackage[nottoc,numbib]{tocbibind} 

\usepackage{natbib} 

\theoremstyle{definition} 

\newtheorem{thm}{Theorem}[section] 
\newtheorem{lemma}[thm]{Lemma}
\newtheorem{prop}[thm]{Proposition}
\newtheorem{cor}[thm]{Corollary}
\newtheorem{ex}[thm]{Example}
\newtheorem{defn}[thm]{Definition}
\newtheorem{rmk}[thm]{Remark}
\newtheorem{note}[thm]{Notation}
\newtheorem{discMM}[thm]{Discussion}

\newcommand{\ds}{\displaystyle}

\makeatletter
\newcommand{\colim@}[2]{%
  \vtop{\m@th\ialign{##\cr
    \hfil$#1\operator@font colim$\hfil\cr
    \noalign{\nointerlineskip\kern1.5\ex@}#2\cr
    \noalign{\nointerlineskip\kern-\ex@}\cr}}%
}
\newcommand{\colim}{%
  \mathop{\mathpalette\colim@{\rightarrowfill@\textstyle}}\nmlimits@
}
\makeatother

\newcommand{\Coeq}{%
  \mathop{\operatorfont Coeq}}

  {\begin{list}{}%
          {\setlength{\leftmargin}{#1}}%
          \item[]%
  }
  {\end{list}}

\newcommand{\oldthm}[3]{\vspace{\topsep} 
\noindent \textbf{#1 \ref{#2}.} #3

\vspace{\topsep}}

\begin{document}
\title{Covers in the Canonical Grothendieck Topology}
\author{C. Lester}
\date{\today}
\maketitle

\begin{abstract}
We explore the canonical Grothendieck topology in some specific circumstances.
First we use a description of the canonical topology to get a variant of Giraud's Theorem. 
Then we explore the canonical Grothendieck topology on the categories of sets and topological spaces; here we get a nice basis for the topology.
Lastly, we look at the canonical Grothendieck topology on the category of $R$-modules.
\end{abstract}

\tableofcontents

\input{intro}

\input{background}
\input{Giraud}

\input{cantopTopSet}
\input{cantopRMod}

\bibliographystyle{plain}
\bibliography{ExamplesCanonicalTopology}
\end{document}

%% file: intro.tex
\section{Introduction}\label{section intro}

In SGA 4.2.2 Verdier defined the canonical Grothendieck topology as the largest Grothendieck topology where all representable presheaves are sheaves.
This paper grew out of an attempt to obtain a precise description of the covers in this Grothendieck topology in the cases of some familiar categories;
we investigate the question for sets, abelian groups, $R$-modules, topological spaces and compactly generated Hausdorff spaces.
The category of sets is simple enough that we can give a complete answer, and in the two categories of topological spaces we give a fairly precise description.
The question for abelain groups and $R$-modules seems to be very subtle, though, and we have only been able to obtain partial results.
Along the way we prove that the canonical topology has a natural appearance in Giraud's Theorem, which is the source for some of our interest in it.
\bigskip

Sieves will be of particular importance in this paper and so we start with a reminder of its definition; we follow the notation and terminology used by Mac Lane and Moerdijk in \cite{maclane}.
For any object $X$ of a category $\EuScript{C}$, we call $S$ a \textit{sieve on $X$} if $S$ is a collection of morphisms, all of whose codomains are $X$, that is closed under precomposition, i.e. if $f\in S$  and $f\circ g$ makes sense, then $f\circ g\in S$.
In particular, we can view a sieve $S$ on $X$ as a full subcategory of the overcategory $(\EuScript{C}\downarrow X)$.

By work from \cite{mineCanTop}, the canonical Grothendieck topology can be characterized in terms of colimits.
Specifically, the canonical Grothendieck topology can be described as the collection of all universal colim sieves where:

\begin{defn}\label{def of cs and ucs}
For a category $\EuScript{C}$, an object $X$ of $\EuScript{C}$ and sieve $S$ on $X$, we call $S$ a {\it colim sieve} if $\colim_{S}{U}$ exists and the canonical map $\colim_{S}{U}\to X$ is an isomorphism.
(Alternatively, $S$ is a colim sieve if $X$ is the universal cocone under the diagram $U\colon S\to \EuScript{C}$.) 
Moreover, we call $S$ a {\it universal colim sieve} if for all arrows $\alpha\colon Y\to X$ in $\EuScript{C}$, $\alpha^\ast S$ is a colim sieve on $Y$.
\end{defn}


One use of this presentation is the following variant of Giraud's Theorem:

\oldthm{Proposition}{Giraud corollary}{If $\EuScript{E}$ is a `nice' category,
then $\EuScript{E}$ is equivalent to the category of sheaves on $\EuScript{E}$ under the canonical topology.}

The universal-colim-sieve presentation also affords us an explicit description of the canonical Grothendieck topology's covers on the category of topological spaces:

\oldthm{Proposition}{Top ucs characterization}{In the category of all topological spaces, $\{A_\alpha\to X\}_{\alpha\in\EuScript{A}}$ is part of a basis for the canonical topology if and only if $\alpha\colon\coprod_{\alpha\in\EuScript{A}} A_\alpha \to X$ is a universal quotient map (i.e.\ $\alpha$ and every pullback of $\alpha$ is a quotient map).
Additionally, a sieve $S$ on $X$ is a (universal) colim sieve if and only if there exists some collection $\{A_\alpha\to X\}_{\alpha\in\EuScript{A}} \subset S$ such that $\coprod_{\alpha\in\EuScript{A}} A_\alpha \to X$ is a (universal) quotient map.
In particular, $T = \langle\{f\colon Y\to X\}\rangle$ is a (universal) colim sieve if and only if $f$ is a (universal) quotient map.}

\vspace{-\topsep}
\oldthm{Proposition}{CGWH ucs characterization}{In the category of compactly generated weakly Hausdorff spaces, $\{A_\alpha\to X\}_{\alpha\in\EuScript{A}}$ is part of the basis for the canonical topology if and only if $\coprod_{\alpha\in\EuScript{A}} A_\alpha \to X$ is a quotient map.
In particular, a sieve $S=\langle \{A_\alpha\to X\}_{\alpha\in \EuScript{A}}\rangle$ on $X$ is in the canonical topology if and only if $\coprod_{\alpha\in\EuScript{A}}A_\alpha\to X$ is a quotient map.
Moreover, every colim sieve is universal.}

\noindent Furthermore, this presentation allows us to more easily compute examples and non-examples in the category of topological spaces; for instance,

\oldthm{Example \ref{direct limit top example}/Example}{direct limit top example for CGWH}{Take $\mathbb{R}^n\to\mathbb{R}^{n+1}$ be the closed inclusion map $(x_1,\dots,x_n)\mapsto(x_1,\dots,x_n,0)$ and use $\mathbb{R}^\infty$ to denote the direct limit $\colim_{n\in\mathbb{N}} \mathbb{R}^n$ with maps $\iota_n\colon \mathbb{R}^n\to\mathbb{R}^\infty$.
Then the cover generated by $\{\iota_n\}_{n\in\mathbb{N}}$ is not in the canonical topology for the category of all topological spaces but is in the canonical topology for the category of compactly generated weakly Hausdorff spaces.}

Additionally, we can use the universal-colim-sieve presentation to get a better idea of the canonical Grothendieck topology's covers on the category of $R$-modules. For example,

\oldthm{Proposition}{good example pf}{Let $S$ be the cover generated by $ \{ f_1\colon M_1\to R, f_2\colon M_2\to R\}$ such that $im(f_i) = a_i R$ for $i = 1,2$.
Then $S$ is in the canonical topology on \textbf{$R$-Mod} if and only if $(a_1,a_2) = R$.}

\vspace{-\topsep}
\oldthm{Proposition}{hope?}{Let $R$ be an infinite principal ideal domain.
Let $S$ be the cover generated by $\{g_i\colon R^n \hookrightarrow R^n\}_{i=1}^M \cup \{f_i\colon R^{m_i}\hookrightarrow R^n\ |\ m_i<n \}_{i=1}^N$. If $S$ a cover in the canonical topology on \textbf{$R$-Mod}, then $g_1\oplus\dots\oplus g_M\colon R^{nM}\to R^n$ is a surjection.}

\vspace{-\topsep}
\oldthm{Proposition}{sieves on Z}{Let $S$ be the cover generated by $\{ \mathbb{Z}\xrightarrow{\times a_i} \mathbb{Z} \}_{i=1}^N$. Then $S$ is in the canonical topology on $\mathbb{Z}\textbf{-Mod}$ if and only if $\text{gcd}(a_1,\dots,a_N) = 1$.}

\vspace{-\topsep}
\oldthm{Proposition}{diag matrices}{Let $S$ be the cover generated by $\{ \mathbb{Z}^n \xrightarrow{A_i} \mathbb{Z}^n \}_{i=1}^N$ where $A_i$ is a diagonal matrix with $\det(A_i)\neq 0$. Then there exists a map $\beta\colon \mathbb{Z}\to \mathbb{Z}^n$ such that $\beta^\ast S$ is not a colim sieve in $\mathbb{Z}\textbf{-Mod}$ if and only if $\text{gcd}(\det(A_1),\dots,\det(A_N))$ does not equal $ 1$.}


\bigskip

\noindent\textit{Organization.}

To start this paper we recall some results from \cite{mineCanTop} in Section \ref{section background}.
Then in Section \ref{section Giraud} we review Giraud's theorem and prove our Corollary to Giraud's Theorem, i.e.\ we prove that that every category $\EuScript{C}$, which satisfies some hypotheses, is equivalent to the category of sheaves on $\EuScript{C}$ with the canonical topology.
In Section \ref{UCS ex of sets and top} we briefly discuss the canonical topology on the category of sets before exploring the canonical topology on the category of topological spaces. Specifically, we look at the category of all topological spaces and the category of compactly generated weakly Hausdorff spaces. We are able to refine our description and obtain a basis for the canonical topology; this result reduces the question ``Is this in the canonical topology?'' to the question ``Is a specific map a universal quotient map?''
Since universal quotient maps have been studied in-depth (for example by Day and Kelly in \cite{DayKelly}), this reduction becomes our most computationally agreeable description of the canonical topology and hence we use it to find some specific examples and non-examples.
Lastly, in Section \ref{UCS in R mod} we investigate the canonical topology on the category of $R$-modules and the category of abelian groups, where we work towards refining our description by making some reductions and obtaining some exclusionary results.
While these reductions and results lead us to some specific examples and non-examples, a basis for the canonical topology remains elusive.
\bigskip

\noindent\textit{General Notation.}

\begin{note}\label{forgetful functor}
For any subcategory $S$ of $(\EuScript{C}\downarrow X)$, we will use $U$ to represent the forgetful functor $S\to\EuScript{C}$. For example, for a sieve $S$ on $X$, $U(f)=\text{domain}\ f$.
\end{note}


\begin{note}\label{generators}
We say that a sieve $S$ on $X$ is \textit{generated} by the morphisms $ \{f_\alpha\colon A_\alpha\to X\}_{\alpha\in\mathcal{A}}$ and write $S = \langle \{f_\alpha\colon A_\alpha\to X\}_{\alpha\in\mathcal{A}}\rangle$ if each $f\in S$ factors through one of the $f_\alpha$, i.e. if $f\in S$ then there exists an $\alpha\in\mathcal{A}$ and morphism $g$ such that $f = f_\alpha\circ g$.
\end{note}
\bigskip

\noindent\textit{Acknowledgements.}

This work is part of the author's doctoral dissertation at the University of Oregon.
The author is extremely grateful to their advisor, Dan Dugger, for all of his guidance, wisdom and patience.

%% file: background.tex
\section{Background}\label{section background} 

This section contains a review of the results from \cite{mineCanTop} that will be used in this paper.

\begin{lemma}\label{pb sieve gen set}
Suppose $\EuScript{C}$ is a category with all pullbacks. \\ Let $S = \langle \{g_\alpha\colon A_\alpha\to X \}_{\alpha\in\mathfrak{A}} \rangle$ be a sieve on object $X$ of $\EuScript{C}$ and $f\colon Y\to X$ be a morphism in $\EuScript{C}$. Then $f^\ast S = \langle \{A_\alpha\times_X Y \overset{\pi_2}{\longrightarrow} Y\}_{\alpha\in\mathfrak{A}} \rangle$.
\end{lemma}

\begin{prop}\label{colim is coeq}
Let $\EuScript{C}$ be a cocomplete category. 
For a sieve in $\EuScript{C}$ on $X$ of the form $S = \langle \{f_\alpha\colon A_\alpha\to X \}_{\alpha\in\mathfrak{A}} \rangle$ such that $A_i\times_X A_j$ exists for all $i,j\in\mathfrak{A}$,
$$\colim_{S}{U} \cong
\Coeq \left(
\begin{tikzcd}\displaystyle
\coprod_{(i,j)\in\mathfrak{A}\times\mathfrak{A}} A_i\times_X A_j \arrow[d, shift right=2] \arrow[d, shift left=2] \\
\displaystyle \coprod_{k\in\mathfrak{A}} A_k
\end{tikzcd}\right)
$$
where the left and right vertical maps are induced from the projection morphisms $\pi_1\colon A_i\times_X A_j \to A_i$ and $\pi_2\colon A_i\times_X A_j \to A_j$.
\end{prop}

\begin{lemma}\label{cs nice with isom}
Let $\EuScript{C}$ be a category. Then $S$ is a colim sieve on $X$ if and only if $f^\ast S$ is a colim sieve for any isomorphism $f\colon Y\to X$.
\end{lemma}

Recall that a morphism $f\colon Y\to X$ is called an \textit{effective epimorphism} provided $Y\times_X Y$ exists, $f$ is an epimorphism and $c\colon \Coeq\left(Y\times_X Y\ \substack{\longrightarrow\\ \longrightarrow}\ Y\right)\to X$ is an isomorphism.
Note that this third condition actually implies the second because $f=c\circ g$ where $g\colon Y\to \Coeq\left(Y\times_X Y\ \substack{\longrightarrow\\ \longrightarrow}\ Y\right)$ is the canonical map. Indeed, $g$ is an epimorphism by an easy exercise and $c$ is an epimorphism since it is an isomorphism.

Additionally, $f\colon Y\to X$ is called a \textit{universal effective epimorphism} if $f$ is an effective epimorphism with the additional property that for every pullback diagram
\begin{center}
\begin{tikzcd}
W \arrow{r} \arrow{d}[left]{\pi_g} &
Y \arrow{d}[right]{f} \\
Z \arrow{r}[below]{g} &
X
\end{tikzcd}
\end{center}
$\pi_g$ is also an effective epimorphism.


\begin{cor}\label{eff epi's gen colim sieves}
Let $\EuScript{C}$ be a cocomplete category with pullbacks.
If $$S = \langle \{f\colon Y\to X\}\rangle$$ is a sieve on $X$, then $S$ is a colim sieve if and only if $f$ is an effective epimorphism. Moreover, $S$ is a universal colim sieve if and only if $f$ is a universal effective epimorphism.
\end{cor}

\begin{thm}\label{ucs is a top}
Let $\EuScript{C}$ be any category. The collection of all universal colim sieves on $\EuScript{C}$ forms a Grothendieck topology.
\end{thm}

\begin{thm}\label{can top is ucs}
For any (locally small) category $\EuScript{C}$, the collection of all universal colim sieves on $\EuScript{C}$ is the canonical topology.
\end{thm}

\begin{prop}\label{reducing sieve generating set}
Let $\EuScript{C}$ be a cocomplete category with pullbacks.
Futher assume that coproducts and pullbacks commute in $\EuScript{C}$.
Then a sieve of the form $S = \langle\{f_\alpha\colon A_\alpha\to X\}_{\alpha\in\EuScript{A}}\rangle$ is a (universal) colim sieve if and only if the sieve $T = \langle\{\coprod f_\alpha\colon \coprod_{\alpha\in\EuScript{A}} A_\alpha\to X\}\rangle$ is a (universal) colim sieve.
\end{prop}

\begin{thm}\label{hinting at basis for can top in special case}
Let $\EuScript{C}$ be a cocomplete category with pullbacks whose coproducts and pullbacks commute.
A sieve $S$ on $X$ is a (universal) colim sieve of $\EuScript{C}$ if and only if there exists some $\{A_\alpha\to X\}_{\alpha\in\EuScript{A}} \subset S$ where $\ds\coprod_{\alpha\in\EuScript{A}} A_\alpha \to X$ is a (universal) effective epimorphism.
\end{thm}

\begin{thm}\label{basis for can top in special case}
Let $\EuScript{C}$ be a cocomplete category with stable and disjoint coproducts and all pullbacks. 
For each $X$ in $\EuScript{C}$, define $K(X)$ by
$$\{A_\alpha\to X\}_{\alpha\in\EuScript{A}}\in K(X)\iff \coprod_{\alpha\in\EuScript{A}} A_\alpha \to X \text{ is a universal effective epimorphism.}$$
Then $K$ is a Grothendieck basis and generates the canonical topology on $\EuScript{C}$.
\end{thm}


%% file: Giraud.tex
\section{Giraud's Theorem and the Canonical Topology}\label{section Giraud}

Giraud's Theorem shows that categories with certain nice properties can be written as sheaves on a Grothendieck site.
We show that in fact, modulo universe considerations, one may take this site to be the original category with the canonical topology.


We will specifically use the version of Giraud's Theorem stated in \cite{maclane}.
In fact, the appendix of \cite{maclane} has a thorough discussion of Giraud's theorem and all of the terminology used in it; we will include the basics of this discussion for completeness. 
We will begin by recalling the definitions used in Mac Lane and Moerdijk's version of Giraud's Theorem.

Throughout this section, let $\EuScript{E}$ be a category with small hom-sets and all finite limits.
\bigskip

\noindent\textsc{Disjoint and Stable Coproducts}
\bigskip

Let $E_\alpha$ be a family of objects in $\EuScript{E}$ and $E = \amalg_\alpha E_\alpha$.

\begin{defn}
The coproduct $E$ is called \textit{disjoint} if every coproduct inclusion $i_\alpha\colon E_\alpha\to E$ is a monomorphism and, whenever $\alpha\neq\beta$, $E_\alpha\times_E E_\beta$ is the initial object in $\EuScript{E}$.
\end{defn}

\begin{defn}
The coproduct $E$ is called \textit{stable} (under pullback) if for every 
$f\colon D\to E$ in $\EuScript{E}$, the morphisms $j_\alpha$ obtained from the pullback diagrams
\begin{center}
\begin{tikzcd}
D\times_E E_\alpha \arrow{r} \arrow{d}[left]{j_\alpha} &
E_\alpha \arrow{d}[right]{i_\alpha} \\
D \arrow{r}[below]{f} &
E
\end{tikzcd}
\end{center}
%
induce an isomorphism $\coprod_\alpha(D\times_E E_\alpha)\cong D$.
\end{defn}

\begin{rmk}\label{stable implies coproducts and pullbacks commute}
If every coproduct in $\EuScript{E}$ is stable, then the pullback operation $-\times_E D$ ``commutes'' with coproducts, i.e.\ 
$(\coprod_\alpha B_\alpha)\times_E D \cong \coprod_\alpha (B_\alpha\times_E D)$.
\end{rmk}
\bigskip

\noindent\textsc{Coequalizer Morphisms and Kernel Pairs}
\bigskip

\begin{defn}
We call a morphism $f\colon Y\to Z$ in $\EuScript{E}$ a \textit{coequalizer} if there exists some object $X$ and morphisms $\partial_0,\partial_1\colon X\to Y$ such that
$$X\, \substack{\overset{\partial_0}{\longrightarrow} \\ \underset{\partial_1}{\longrightarrow}}\, Y \overset{f}{\longrightarrow} Z$$
%
is a coequalizer diagram.
\end{defn}

We remark that every coequalizing morphism is an epimorphism but the converse of this statement is not guaranteed.

\begin{defn}
The pair of morphisms $\partial_0,\partial_1\colon X\to Y$ are called a \textit{kernel pair} for $f\colon Y\to Z$ if the following is a pullback diagram
\begin{center}
\begin{tikzcd}
X \arrow{r}[above]{\partial_1} \arrow{d}[left]{\partial_0} &
Y \arrow{d}[right]{f} \\
Y \arrow{r}[below]{f} &
Z
\end{tikzcd}
\end{center}
\end{defn}
\bigskip

\noindent\textsc{Equivalence Relations and Quotients}
\bigskip

\begin{defn}
An \textit{equivalence relation} on the object $E$ of $\EuScript{E}$ is a subobject $R$ of $E\times E$, represented by the monomorphism $(\partial_0,\partial_1)\colon R\to E\times E$, satisfying the following axioms
\begin{enumerate}
\item (reflexive) the diagonal $\Delta\colon E\to E\times E$ factors through $(\partial_0,\partial_1)$,
\item (symmetric) the map $(\partial_1,\partial_0)\colon R\to E\times E$ factors through $(\partial_0,\partial_1)$,
\item (transitivity) if $R\times_E R$ is the pullback
\begin{center}
\begin{tikzcd}
R\times_E R \arrow{r}[above]{\pi_1} \arrow{d}[left]{\pi_0} &
R \arrow{d}[right]{\partial_0} \\
R \arrow{r}[below]{\partial_1} &
E
\end{tikzcd}
\end{center}
%
then $(\partial_1\pi_1,\partial_0\pi_0)\colon R\times_E R \to E\times E$ factors through $R$.
\end{enumerate}
\end{defn}

\begin{defn}
If $E$ is an object of $\EuScript{E}$ with equivalence relation $R$, then the \textit{quotient} is denoted $E/R$ and is defined to be
$$\Coeq\left( R\, \substack{\overset{\partial_0}{\longrightarrow} \\ \underset{\partial_1}{\longrightarrow}}\, E \right)$$ provided that this coequalizer exists.
\end{defn}
\bigskip

\noindent\textsc{Stably Exact Forks}
\bigskip

A diagram is called a \textit{fork} if it is of the form
\begin{equation}\label{fork}
X\, \substack{\overset{\partial_0}{\longrightarrow} \\ \underset{\partial_1}{\longrightarrow}}\, Y \overset{q}{\longrightarrow} Z.
\end{equation}

\begin{defn}
The fork (\ref{fork}) is called \textit{exact} if $\partial_0$ and $\partial_1$ are the kernel pair for $q$, and $q$ is the coequalizer of $\partial_0$ and $\partial_1$.
\end{defn}

\begin{defn}
The fork (\ref{fork}) is called \textit{stably exact} if the pullback of (\ref{fork}) along any morphism in $\EuScript{E}$ yields an exact fork, i.e. if for any $Z' \to Z$ in $\EuScript{E}$,
$$X\times_Z Z'\, \substack{\longrightarrow \\ \longrightarrow}\, Y\times_Z Z' \overset{q\times 1}{\longrightarrow} Z\times_Z Z'$$
%
is an exact fork.
\end{defn}
\bigskip

\pagebreak

\noindent\textsc{Generating Sets}
\bigskip

\begin{defn}
A set of objects $\{A_i\,|\,i\in I\}$ of $\EuScript{E}$ is said to \textit{generate} $\EuScript{E}$ if for every object $E$ of $\EuScript{E}$, $W = \{A_i\to E\,|\, i\in I\}$ is an epimorphic family (in the sense that for any two parallel arrows $u,v\colon E\to E'$, if every $w\in W$ yields the identity $uw=vw$, then $u=v$).
\end{defn}

\bigskip

\noindent\textsc{Giraud's Theorem}

\begin{thm}[Giraud, \cite{maclane}]\label{Giraud}
A category $\EuScript{E}$ with small hom-sets and all finite limits is a Grothendieck topos if and only if it has the following properties (which we will refer to as {\it Giraud's axioms}):
\begin{enumerate}
\item[(i)] $\EuScript{E}$ has small coproducts which are disjoint and stable under pullback,
\item[(ii)] every epimorphism in $\EuScript{E}$ is a coequalizer,
\item[(iii)] every equivalence relation $R\ \substack{\to \\ \to}\ E$ in $\EuScript{E}$ is a kernel pair and has a quotient,
\item[(iv)] every exact fork $R\ \substack{\to\\\to}\ E \to Q$ is stably exact,
\item[(v)] there is a small set of objects of $\EuScript{E}$ which generate $\EuScript{E}$.
\end{enumerate}
\end{thm}



\begin{discMM}\label{MM epi}
Taken together, Giraud's axioms (ii) and (iv) imply that for each epimorphism $B\xrightarrow{f} A$, the fork $B\times_A B\ \substack{\to \\ \to}\ B\to A$ is stably exact.
The exactness implies $f$ is an effective epimorphism and the stability implies $f$ is a universal effective epimorphism.
\end{discMM}

\begin{note}
We use $Sh(\EuScript{E},J)$ to represent the category of sheaves on the category $\EuScript{E}$ under the topology $J$.
\end{note}

Suppose the category $\EuScript{E}$ has small hom-sets and all finite limits, satisfies Giraud's axioms, and whose small set of generators (axiom v) is $\EuScript{C}$.
In \cite{maclane} Mac Lane and Moerdijk specifically prove $\EuScript{E}\cong Sh(\EuScript{C},J)$ where $J$ is the Grothendieck topology on $\EuScript{C}$ defined by:
\begin{center} 
$S\in J(X)$ if and only if $\displaystyle \coprod_{(g\colon D\to X)\in S} D\to X$ is an epimorphism in $\EuScript{E}$.
\end{center} 
(In particular, Mac Lane and Moerdijk prove that $J$ is a Grothendieck topology.)

\begin{prop}\label{Giraud corollary}
Suppose the category $\EuScript{E}$ has small hom-sets and all finite limits, satisfies Giraud's axioms, and whose small set of generators (axiom v) is $\EuScript{C}$.
Then $\EuScript{E}$ is equivalent to $Sh(\EuScript{C},C)$ where $C$ is the canonical topology on $\EuScript{C}$.
\end{prop}


\begin{proof}
Let $J$ be the topology defined above. 
Additionally, the above discussion 
implies that it suffices to show that $J$ is the canonical topology.
By Theorem \ref{can top is ucs}, we will instead show that every universal colim sieve is in $J$ and that every sieve in $J$ is a universal colim sieve.

By Remark \ref{stable implies coproducts and pullbacks commute}, coproducts and pullbacks commute and hence for any collection of morphisms $\{A_i\to X\}_{i\in I}$ in $\EuScript{E}$, the diagrams
\begin{center}
\begin{tikzcd}
\coprod_{I^2} (A_i\times_X A_j) \arrow[d, shift right=2] \arrow[d, shift left=2]\\
\coprod_{I} A_k
\end{tikzcd}
\hspace{0.05in} and \hspace{0.05in}
\begin{tikzcd}
\left(\coprod_I A_i \right)\times_X\left(\coprod_I A_j \right) \arrow[d, shift right=2] \arrow[d, shift left=2] \\
\coprod_{I} A_k
\end{tikzcd}
\end{center}
are isomorphic. Note: in both diagrams, the two maps down are the obvious ones induced/obtained from a pullback diagram. 
Thus
$$\Coeq\left(
\begin{tikzcd}
\coprod_{I^2} (A_i\times_X A_j) \arrow[d, shift right=2] \arrow[d, shift left=2]\\
\coprod_{I} A_k
\end{tikzcd}
\right) \cong
\Coeq\left(
\begin{tikzcd}
\left(\coprod_I A_i \right)\times_X\left(\coprod_I A_j \right) \arrow[d, shift right=2] \arrow[d, shift left=2] \\
\coprod_{I} A_k
\end{tikzcd}
\right).$$
But by Proposition \ref{colim is coeq} (which is usable since $\EuScript{E}$ is cocomplete),
$$\Coeq\left(
\begin{tikzcd}
\coprod_{I^2} (A_i\times_X A_j) \arrow[d, shift right=2] \arrow[d, shift left=2]\\
\coprod_{I} A_k
\end{tikzcd}
\right)\cong
\colim_{S}{U} \quad \text{where } S = \left<\{A_i\to X\}_{i\in I}\right>$$ and
$$\Coeq\left(
\begin{tikzcd}
\left(\coprod_I A_i \right)\times_X\left(\coprod_I A_j \right) \arrow[d, shift right=2] \arrow[d, shift left=2] \\
\coprod_{I} A_k
\end{tikzcd}
\right)\cong
\colim_{T_S}{U} \quad \text{where } T_S = \left< \left\{ \left(\coprod_I A_i\right)\to X \right\}\right>.$$
Hence
\begin{equation}\label{fact}
\begin{split}
\colim_{S}{U}&\cong\colim_{T_S}{U} \\
\text{where } S = \left<\{A_i\to X\}_{i\in I}\right> \quad &\text{and} \quad T_S = \left< \left\{ \left(\coprod_I A_i\right)\to X \right\}\right>
\\
\text{for any generating set } & \{A_i\to X\}_{i\in I} \text{ of $S$.}
\end{split}
\end{equation}

Suppose $S$ is a universal colim sieve.
Since $S$ has the some generating set, then by the definition of colim sieve and (\ref{fact}),
$$X\cong \colim_{S}{U}\cong \colim_{T_S}{U}.$$
This implies that $T_S$ is a colim sieve.
Hence $\left(\coprod_{(g\colon D\to X)\in S} D\right) \to X$ is an effective epimorphism by Corollary \ref{eff epi's gen colim sieves} and so $S\in J(X)$.

For the converse, suppose that $S\in J(X)$.
Thus $p_s\colon \left(\coprod_{(g\colon D\to X)\in S} D\right) \to X$ is an epimorphism, which by Discussion \ref{MM epi} 
is a universal effective epimorphism.
Hence by Corollary \ref{eff epi's gen colim sieves}, $p_s$ generates a universal colim sieve called $T_S$.
Then by the definition of colim sieve and (\ref{fact}),
$$X\cong \colim_{T_S}{U} \cong \colim_{S}{U}.$$
Therefore $S$ is a colim sieve.

Similar to the last paragraph, we can use (\ref{fact})
to show that $f^\ast S$ is a colim sieve for any morphism $f$ in $\EuScript{E}$ if we know that $T_{f^\ast S}$ is a colim sieve.
So to finish the proof we will use the fact that $T_S$ is a universal colim sieve to show that $T_{f^\ast S}$ is a colim sieve.
Let $f\colon Y\to X$ be any morphism in $\EuScript{E}$.
Then by using $S$ as a generating collection for itself and Lemma \ref{pb sieve gen set}, $f^\ast S = \left< \{ A\times_X Y\to Y\ |\ A\to X\in S \} \right>$.
Similarly, using Lemma \ref{pb sieve gen set}, $f^\ast T_S = \left<\left\{ \left(\coprod_{(A\to X\in S)} A\right)  \times_X Y\to Y\right\}\right>$.
Then by Remark \ref{stable implies coproducts and pullbacks commute} 
$$\displaystyle \coprod_{(A\to X)\in S} (A\times_X Y) \cong \left(\coprod_{(A\to X) \in S} A\right)  \times_X Y$$
over $Y$.
Therefore,
$$\colim_{T_{f^\ast S}}{U}\cong \colim_{f^\ast T_S}{U} \cong Y$$
where the first isomorphism is due to the previous few sentences
and the second isomorphism is due to the fact that $T_S$ is a universal colim sieve.
Thus $T_{f^\ast S}$ is a colim sieve. 
\end{proof}

%% file: cantopTopSet.tex
\section{Universal Colim Sieves in the Categories of Sets and Topological Spaces}\label{UCS ex of sets and top}

In this section we examine the canonical topology on the categories of sets, all topological spaces and compactly generated weakly Haudsdorff spaces.


\begin{note}
We will use \textbf{Sets} to denote the category of sets. We will use \textbf{Top} to denote the category of all topological spaces, \textbf{CG} to denote the category of compactly generated spaces, and \textbf{CGWH} to denote the category of compactly generated weakly Hausdorff spaces. When we want to talk about the category of topological spaces without differentiating between \textbf{Top} and \textbf{CGWH}, then we will use \textbf{Spaces}; all results about \textbf{Spaces} will hold for both \textbf{Top} and \textbf{CGWH}.
\end{note}

We will begin with a few reminders about the category of compactly generated weakly Hausdorff spaces based on the references \cite{strickland} and \cite{may}.
Specifically, there are functors $k\colon\textbf{Top}\to\textbf{CG}$ and $h\colon\textbf{CG}\to\mathbf{CGWH}$ such that
\begin{itemize}
\item For a topological space $X$ with topology $\tau$, a subset $Y$ of $X$ is called \text{$k$-closed} if $u^{-1}(Y)$ is closed in $K$ for every continuous map $u\colon K\to X$ and compact Hausdorff space $K$. The collection of all $k$-closed subsets, called $k(\tau)$, is a topology.
\item The functor $k$ takes $X$ with topology $\tau$ to the set $X$ with topology $k(\tau)$.
\item $k$ is right adjoint to the inclusion functor $\iota\colon\textbf{CG}\to\textbf{Top}$.
\item $h(X)$ is $X/E$ where $E$ is the smallest equivalence relation on $X$ closed in $X\times X$.
\item $h$ is left adjoint to the inclusion functor $\iota'\colon\textbf{CGWH}\to\textbf{CG}$. 
\item A limit in $\textbf{CGWH}$ is $k$ applied to the limit taken in $\textbf{Top}$, i.e. for a diagram $F\colon I\to \textbf{CGWH}$, the limit of $F$ is $k(\lim_{I} \iota \iota' F)$.
\item A colimit in $\textbf{CGWH}$ is $h$ applied to the colimit taken in $\textbf{Top}$, i.e. for a diagram $F\colon I\to \textbf{CGWH}$, the colimit of $F$ is $h(\colim_{I} \iota \iota' F)$.
\end{itemize}

\begin{prop}\label{top colims inj}
Let $S$ be a sieve on $X$ in either \textbf{Sets} or \textbf{Top}. Let $C$ be $\ds\colim_{S}{U}$. Then the natural map $\varphi\colon C \to X$ is an injection.
\end{prop}

\begin{proof}
Suppose $\tilde{y},\tilde{z}\in C$ and $\varphi(\tilde{y}) = x = \varphi(\tilde{z})$. We can pick a $(Y\to X)\in S$ and a $y\in Y$ that represents $\tilde{y}$, i.e. where $y\mapsto \tilde{y}$ under the natural map $Y\to C$; similarly, we can pick a $(Z\to X)\in S$ and a $z\in Z$ representing $\tilde{z}$.
Then the inclusion $i\colon  \{ x \} \hookrightarrow X$ factors through both $Y$ and $Z$ by $x\mapsto y$ and $x\mapsto z$ respectively. Thus $i\in S$. Hence $\tilde{y}=\tilde{z}$ in $C$.
\end{proof}

\begin{cor}\label{cgwh colims inj}
Let $S$ be a sieve on $X$ in \textbf{CGWH}. Then the colimit over $S$ taken in \textbf{Top} is in \textbf{CGWH}, i.e. $h(\colim_{I} \iota \iota' U) = \colim_{I} \iota \iota' U$. Moreover, the natural map $\varphi\colon \colim_{S}{U} \to X$ is an injection.
\end{cor}

\begin{proof}
We will make use of the following Proposition from \cite{strickland}: if $Z$ is in \textbf{CG}, then $Z$ is weakly Hausdorff if and only if the diagonal subspace $\Delta_Z$ is closed in $Z\times Z$.
Additionally, we remark that colimits of compactly generated spaces computed in \textbf{Top} are automatically compactly generated.

Let $C = \colim_{S} \iota\iota' U$, i.e. $C$ is the colimit over $S$ taken in \textbf{Top}.
By Proposition \ref{top colims inj}, the natural map $\varphi\colon C\to X$ is an injection; we remark that it is not the statement of Proposition \ref{top colims inj} that gives this observation since $S$ is not a sieve in \text{Top}, instead the proof of Proposition \ref{top colims inj} holds in this situation since  $\{x\}$ is in \textbf{CGWH}.
Since $X$ is \textbf{CGWH}, then $\Delta_X$ is closed in $X\times X$.
Since $\varphi$ is a continuous injection, then $(\varphi\times\varphi)^{-1}(\Delta_X) = \Delta_C$ is closed in $C\times C$.
\end{proof}


\subsection{Basis and Presentation}\label{section basis for Top and Sets}



The categories \textbf{Sets}, \textbf{Top} and \textbf{CGWH} all satisfy the hypotheses of Theorems \ref{basis for can top in special case} and \ref{hinting at basis for can top in special case}. Thus we have the following corollaries of Theorems \ref{basis for can top in special case} and \ref{hinting at basis for can top in special case} based on what the universal effective epimorphisms are in each category.

\begin{prop}\label{can in sets}
In \textbf{Sets}, $\{A_\alpha\to X\}_{\alpha\in\EuScript{A}}$ is part of a basis for the canonical topology if and only if $\coprod_{\alpha\in\EuScript{A}} A_\alpha \to X$ is a surjection.
In particular, a sieve of the form $S=\langle \{A_\alpha\to X\}_{\alpha\in \EuScript{A}}\rangle$ on $X$ is in the canonical topology if and only if $\ds\coprod_{\alpha\in\EuScript{A}}A_\alpha\to X$ is a surjection.
Moreover, every colim sieve is universal.
\end{prop}

\begin{proof}
It is easy to see in \textbf{Sets} that the effective epimorphisms are precisely the surjections.
Since pulling back a surjection yields a surjection, then the universal effective epimorphisms in the category of sets are also the surjections.
Lastly, this implies, by Theorem \ref{hinting at basis for can top in special case}, that every colim sieve is universal.
\end{proof}

\begin{rmk}
Since \textbf{Sets} is a Grothendieck topos, we can compare Proposition \ref{can in sets} to the proof of Proposition \ref{Giraud corollary}. 
Specifically, Proposition \ref{can in sets} allows us to determine if a sieve is in the canonical topology by looking only at the sieve's generating set whereas the proof of Proposition \ref{Giraud corollary} 
along with the Grothendieck topology $J$ require us to look at the entire sieve.
\end{rmk}

Recall that a quotient map $f$ is called \textit{universal} if every pullback of $f$ along a map yields a quotient map.

\begin{prop}\label{Top ucs characterization}
In \textbf{Top}, $\{A_\alpha\to X\}_{\alpha\in\EuScript{A}}$ is part of a basis for the canonical topology if and only if $\coprod_{\alpha\in\EuScript{A}} A_\alpha \to X$ is a universal quotient map.
Additionally, a sieve $S$ on $X$ is a (universal) colim sieve if and only if there exists some collection $\{A_\alpha\to X\}_{\alpha\in\EuScript{A}} \subset S$ such that $\ds\coprod_{\alpha\in\EuScript{A}} A_\alpha \to X$ is a (universal) quotient map.
In particular, $T = \langle\{f\colon Y\to X\}\rangle$ is a (universal) colim sieve if and only if $f$ is a (universal) quotient map.
\end{prop}

\begin{proof}
It is a well-known fact that in \textbf{Top} the effective epimorphisms are precisely the quotient maps.
\end{proof}

\begin{prop}\label{CGWH ucs characterization}
In \textbf{CGWH}, $\{A_\alpha\to X\}_{\alpha\in\EuScript{A}}$ is part of the basis for the canonical topology if and only if $\coprod_{\alpha\in\EuScript{A}} A_\alpha \to X$ is a quotient map.
In particular, a sieve $S=\langle \{A_\alpha\to X\}_{\alpha\in \EuScript{A}}\rangle$ on $X$ is in the canonical topology if and only if $\ds\coprod_{\alpha\in\EuScript{A}}A_\alpha\to X$ is a quotient map.
Moreover, every colim sieve is universal.
\end{prop}

\begin{proof}
This is a consequence of Corollary \ref{eff epi's gen colim sieves}, Corollary \ref{cgwh colims inj}, the fact that the universal effective epimorphisms in \textbf{Top} are precisely the universal quotient maps, and \cite[][Proposition 2.36]{strickland}, which states that every quotient map in \textbf{CGWH} is universal.
\end{proof}
\subsection{Examples in the category of Spaces}

In this section we will use our basis to talk about some specific examples; including a special circumstance (when a sieve is generated by one function) and how the canonical topology on the categories \textbf{CGWH} and \textbf{Top} can differ in this situation.

\begin{defn}
For a category $D$, we call $\mathfrak{A}\subset \operatorfont{ob}(D)$ a \textit{weakly terminal set} of $D$ if for every object $X$ in $D$, there exists some $A\in\mathfrak{A}$ and morphism $X\to A$ in $D$.

Additionally, if $F\colon D\to C$ is a functor and $D$ has a weakly terminal set $\mathfrak{A}$, then we call $\{F(A)\}_{A\in\mathfrak{A}}$ a \textit{weakly terminal set} of $F$. 
\end{defn}

For example, if $S = \langle \{A_\alpha\to X \}_{\alpha\in\mathfrak{A}} \rangle$ is a sieve on $X$ then $\{A_\alpha\}_{\alpha\in\mathfrak{A}}$ is the weakly terminal set of $U$.
Or as another example, $\{Y\}$ is the weakly terminal set of the diagram $Y \times_X Y\ \substack{\longrightarrow\\ \longrightarrow}\ Y$.
One easy consequence of this in \textbf{Top} is a reduction of the colimit topology: $V$ is open in the colimit if and only if the preimage of $V$ is open in each member of the weakly terminal set.



\begin{prop}\label{suff cond to get open colim maps}
Let $F\colon D \to \textbf{Spaces}$ be a functor where $D$ has a weakly terminal set $\mathfrak{A}$. Suppose $f_A\colon F(A)\to X$ is an open map for all $A\in\mathfrak{A}$, then the induced map $\varphi\colon \colim_{D} F\to X$ is an open map. Similarly, if the $f_A$ are all closed and $\mathfrak{A}$ is a finite set, then $\varphi$ is a closed map.
\end{prop}

\begin{proof}
Let $C= \colim F$ and $i_A\colon F(A)\to C$ be the natural maps. Both results follow from the easy set equality below for $B\subset C$
$$\varphi(B) = \bigcup_{A\in\mathfrak{A}} f_A(i_A^{-1}(B))$$
since $i_A^{-1}$, $f_A$ and unions respect open/closed sets in their respective scenarios.
\end{proof}

\begin{cor}\label{opens/closeds gen colim sieves}
Let $S=\langle \{f_\alpha\colon A_\alpha\to X\}_{\alpha \in \EuScript{A}} \rangle$ be a sieve on $X$ in \textbf{Spaces} with the induced map $\eta\colon \ds\coprod_{\alpha\in\EuScript{A}} A_\alpha \to X$ a surjection. If all of the $f_\alpha$ are open maps or if $\EuScript{A}$ is a finite collection and all of the $f_\alpha$ are closed maps, then $S$ is a colim sieve.
\end{cor}

\begin{proof}
Let $\varphi\colon \ds\colim_{S}{U}\to X$ be the natural map. By Proposition \ref{top colims inj}, Corollary \ref{cgwh colims inj}, and the surjectivity of $\eta$, 
$\varphi$ is a continuous bijection. Then Proposition \ref{suff cond to get open colim maps} implies that $\varphi$ is open or closed, depending on the case, and hence an isomorphism.
\end{proof}

This corollary leads us to some nice examples of sieves we would hope are in the canonical topology and actually are!

\begin{ex}\label{open cover can top ex}
Let $X$ be any space and let $\{U_i\}_{i\in I}$ be an open cover of $X$. Then the inclusion maps $U_i\hookrightarrow X$ generate a universal colim sieve, call it $S$. Indeed, by Corollary \ref{opens/closeds gen colim sieves}, $S$ is a colim sieve.
Universality is obvious, as the preimage of an open cover is an open cover.
\end{ex}

\begin{ex}
Let $X$ be any space and let $K_1, \dots, K_n$ be a closed cover of $X$. For the exact same reasons as the previous example, the inclusions $K_i\hookrightarrow X$ generate a sieve in the canonical topology.
\end{ex}


Before we give our next example, we rephrase \cite[][Theorem 1]{DayKelly}, which completely characterizes universal quotient maps in \textbf{Top}:

\begin{thm}[Day and Kelly, 1970]\label{Day and Kelly}
Let $f\colon Y\to X$ be a quotient map. Then $f$ is a universal quotient map if and only if for every $x\in X$ and cover $\{G_\alpha\}_{\alpha\in\Lambda}$ of $f^{-1}(x)$ by opens in $Y$, there is a finite set $\{\alpha_1, \dots, \alpha_n\}\subset\Lambda$ such that $fG_{\alpha_1}\cup\dots\cup fG_{\alpha_n}$ is a neighborhood of $x$.
\end{thm}

\begin{ex}\label{direct limit top example}
Consider the diagram $B_1\to B_2\to B_3\to \dots$ and the direct limit $B=\colim B_n$ in \textbf{Top}.
Let $S = \langle\{\iota_n\colon B_n\to B\,|\, n\in\mathbb{N}\}\rangle$ where $\iota_n$ are the natural maps into the colimit.
By Proposition \ref{Top ucs characterization}, $S$ is a colim sieve because $\coprod_{n\in\mathbb{N}}B_n\to B$ is obviously a quotient map.
However, $S$ is not necessarily in the canonical topology -- we can use Proposition \ref{Top ucs characterization} on specific examples to see when $S$ is and is not in the canonical topology.

For example, suppose there exists an $N$ such that $B_m = B_N$ whenever $m>N$.
Then $B = B_N$.
Hence it is easy to see by Day and Kelly's condition that the map $\coprod_{n\in\mathbb{N}}B_n\to B$ is a universal quotient map.
Therefore, the $S$ from this example is in the canonical topology.

As another example, take $B_n = \mathbb{R}^n$ and let $B_n\to B_{n+1}$ be the closed inclusion map $(x_1,\dots,x_n)\mapsto(x_1,\dots,x_n,0)$.
Use $\mathbb{R}^\infty$ to denote the direct limit.
We claim that $\coprod_{n\in\mathbb{N}}\mathbb{R}^n\to\mathbb{R}^\infty$ is not a universal quotient map.
Indeed, consider Day and Kelly's condition; take $x = 0\in\mathbb{R}^\infty$ and the open cover in $\coprod_{n\in\mathbb{N}}\mathbb{R}^n$ consisting of open disks $D^n\subset\mathbb{R}^n$ centered at the origin with fixed radius $\epsilon>0$.
Pick any finite collection $D^{n_1},\dots,D^{n_k}$ with $n_1<\dots<n_k$.
Then for $i=1,\dots,k$ we can view $D^{n_i}$ as a subset of $\mathbb{R}^{n_k}$.
Hence $\cup_{i=1}^k \iota_{n_i}(D^{n_i})$ is $\cup_{i=1}^k \iota_{n_k}(D^{n_i})\subset \iota_{n_k}(\mathbb{R}^{n_k}$).
However, by dimensional considerations, we can see that for all $b\in\mathbb{N}$, $\iota_b(\mathbb{R}^b)$ contains no open sets of $\mathbb{R}^\infty$ and hence $\cup_{i=1}^k \iota_{n_i}(D^{n_i})$ cannot be a neighborhood of $x$ in $\mathbb{R}^\infty$.
Remark: To see that $\iota_b(\mathbb{R}^b)$ contains no open sets, suppose to the contrary and call the open set $V$.
Then $\iota_{b+1}^{-1}(V)$ is open in $\mathbb{R}^{b+1}$ and in particular, contains an open ball of dimension $b+1$.
Thus dimensional considerations imply that $\iota_{b+1}^{-1}(V)$ is not contained in the image of $\mathbb{R}^b$ in $\mathbb{R}^{b+1}$.
Since each $\iota_n$ is an inclusion map, then $\iota_{b+1}\iota_{b+1}^{-1}(V)\not\subset \iota_{b+1}(\mathbb{R}^b)$ and so
$V$ is not contained in $\iota_b(\mathbb{R}^b)$, which is our contradiction.
Therefore, the $S$ from this example is not in the canonical topology.
\end{ex}

\begin{ex}\label{direct limit top example for CGWH}
Consider the diagram $B_1\to B_2\to B_3\to \dots$ and the direct limit $B=\colim B_n$ in \textbf{CGWH}.
Let $S = \langle\{\iota_n\colon B_n\to B\,|\, n\in\mathbb{N}\}\rangle$ where $\iota_n$ are the natural maps into the colimit.
Then by Proposition \ref{CGWH ucs characterization}, $S$ is a universal colim sieve because $\coprod_{n\in\mathbb{N}}B_n\to B$ is a quotient map.
\end{ex}

Now we shift our focus to sieves that can be generated by one map, called \textit{monogenic sieves}.
There are many reasons one could focus on these kinds of sieves, however by Proposition \ref{reducing sieve generating set}, 
if we fully comprehend when monogenic sieves are in the canonical topology, then we can (in some sense) completely understand the canonical topology.
From this point onward, this section will be about monogenic sieves; in other words, by Proposition \ref{Top ucs characterization} and Proposition \ref{CGWH ucs characterization}, we will be focusing on (universal) quotient maps.

\begin{rmk}
Some examples will talk about the space $\mathbb{R}/\mathbb{Z}$.
In this section, this space is not a group quotient but instead is the squashing of the subspace $\mathbb{Z}$ to a point.
\end{rmk}

\begin{ex}
Consider the quotient maps $f\colon S^n\to \mathbb{R}P^n$ and $g\colon \mathbb{R}\to \mathbb{R}/\mathbb{Z}$.
There is some subtly, which will depend on the category we are in, in determining if $f$ or $g$ generate universal colim sieves.
Throughout the rest of this section we will continue to explore this particular example. 
\end{ex}


\bigskip

\noindent\textsc{Monogenic Sieves in \textbf{CGWH}}
\bigskip

By Proposition \ref{CGWH ucs characterization}, if $X$ and $Y$ are in \textbf{CGWH} and $h\colon Y\to X$, then $\langle\{h\}\rangle$ is in the canonical topology if and only if $h$ is a quotient map.
Therefore, we immediately get the following examples:

\begin{ex}
Topological manifolds are in \textbf{CGWH}. Thus $S^n$ and $\mathbb{R}P^n$ are in \textbf{CGWH}. Hence $\langle\{ f\colon S^n\to \mathbb{R}P^n \}\rangle$ is in the canonical topology. 
\end{ex}

\begin{ex}
Every CW-complex is in \textbf{CGWH}. Thus $\mathbb{R}$ and $\mathbb{R}/\mathbb{Z}$ are in \textbf{CGWH}. Hence $\langle\{ g\colon  \mathbb{R}\to\mathbb{R}/\mathbb{Z}\}\rangle$ is in the canonical topology. 
\end{ex}
\bigskip

\noindent\textsc{Monogenic Sieves in \textbf{Top}}
\bigskip

This section will heavily rely on Theorem \ref{Day and Kelly} (the Theorem by Day and Kelly characterizing universal quotient maps in \textbf{Top}) because a monogenic sieve generated by $f$ is in the canonical topology if and only if $f$ is a universal quotient map.

\begin{ex}
Day and Kelly's theorem implies that every open quotient map is a universal quotient map.
Therefore, the quotient map $f\colon S^n\to \mathbb{R}P^n$ is a universal quotient map and $\langle \{ f\colon S^n\to \mathbb{R}P^n \}\rangle$ is in the canonical topology.
\end{ex}

\begin{ex}
The quotient map $g\colon \mathbb{R}\to \mathbb{R}/\mathbb{Z}$ is not universal. We will demontrate this in two ways, first by using Day and Kelly's theorem and second by directly showing $g$ is not universal. Note: many sets of $\mathbb{R}/\mathbb{Z}$ will be written as if they are in $\mathbb{R}$ for ease of presentation.

(i) We will look at Day and Kelly's condition for $\mathbb{Z}\in \mathbb{R}/\mathbb{Z}$ with the open cover (in $\mathbb{R}$) $\{G_i \coloneqq (i-m,i+m)\}_{i\in\mathbb{Z}}$ for a fixed $m\in\left(0,\frac{1}{2}\right)$.
For any open set $U$ of $\mathbb{R}/\mathbb{Z}$ containing $\mathbb{Z}$, the quotient topology tells us that $g^{-1}(U)$ is an open neighborhood of $\mathbb{Z}\subset\mathbb{R}$.
But for any $n$, $g^{-1}(\bigcup_{k=1}^n gG_{i_k}) = \mathbb{Z} \cup \left(\bigcup_{k=1}^n (i_k-m, i_k+m)\right)$ is not a neighborhood of $\mathbb{Z}\subset\mathbb{R}$.
So there cannot be any open set of $\mathbb{R}/\mathbb{Z}$ containing $\mathbb{Z}$ that is contained in $\bigcup_{k=1}^n gG_{i_k}$ for any finite collection of the cover.

(ii) To directly show that $g$ is not universal we need to come up with a space and map to $\mathbb{R}/\mathbb{Z}$ where $g$ pulledbacked along this map is not a quotient map. Our candidate is the following:
Let $t(\mathbb{R}/\mathbb{Z})$ be the set $\mathbb{R}/\mathbb{Z}$ with the topology where $U$ (written as if it is in $\mathbb{R})$ is said to be open if (a) $\mathbb{Z}\not\subset U$ or (b) $U$ contains $\mathbb{Z}$ and is a neighborhood (in the typical topology) of $(\mathbb{Z}-\{\text{finitely many or no points}\})$.
Remark: this topology was used in Day and Kelly's paper (in the proof of their theorem), however they defined the topology using a filter and we have merely rephrased it for convenience.

Define $\kappa\colon t(\mathbb{R}/\mathbb{Z}) \to \mathbb{R}/\mathbb{Z}$ by the set identity map; this is a continuous map.
As a set, the pullback of $\text{domain}(g)$ along $\kappa$ is $\mathbb{R}$ but since it now has the limit topology, we denote the pullback as $t(\mathbb{R})$; in particular, $t(\mathbb{R})$ is $\mathbb{R}$ with the discrete topology.
Denote the projection maps as $g'\colon  t(\mathbb{R})\to t(\mathbb{R}/\mathbb{Z})$ and $\kappa'\colon  t(\mathbb{R}) \to \mathbb{R}$.

We claim that $g'$ is not a quotient map, i.e. there is some non-open set $B$ in $t(\mathbb{R}/\mathbb{Z})$ with $(g')^{-1}(B)$ open in $t(\mathbb{R})$.
Since every $(g')^{-1}(B)$ is open in $t(\mathbb{R})$, then we merely need to find a $B$ that is not open in $t(\mathbb{R}/\mathbb{Z})$; $B = \{\mathbb{Z}\}$ obviously works.
\end{ex}

The above example shows us that quotient maps of the form $X\to X/A$ may not generate universal colim sieves. 
So let's understand these special quotient maps a little better.
Specifically, using Day and Kelly's theorem, we can completely state what kinds of subspaces $A$ yield universal quotient maps $X\to X/A$:

\begin{cor}\label{DK corollary}
The quotient map $\pi\colon X\to X/A$ is universal if and only if both of the following properties hold:
\begin{enumerate}
\item If $A$ is not open, then for every open cover $\{G_\alpha\}_{\alpha \in \Lambda}$ of $(\partial A)\cap A$ in $X$ there is a finite collection $\{\alpha_1, \dots, \alpha_n\} \subset \Lambda$ with $A\cup G_{\alpha_1}\cup \dots\cup G_{\alpha_n}$ open in $X$.
\item If $A$ is not closed, then for every open $U$ in $X$ such that $U\cap (\overline{A}-A)\neq \emptyset$, $U\cup A$ is open in $X$.
\end{enumerate}
\end{cor}

\begin{proof}
We will be using Theorem \ref{Day and Kelly} in two ways: first by finding the necessary conditions for $\pi$ to be a universal quotient map (i.e. proving the forward direction) and then second by checking the sufficient conditions in the three cases (i) $x = A$, (ii) $x\in X-\overline{A}$, and (iii) $x\in \overline{A}-A$ (i.e. proving the backward direction).

First suppose that $\pi$ is a universal quotient map.
To see that the first property is necessary, assume that $(\partial A)\cap A\neq \emptyset$, i.e. $A$ is not open, and we have an open cover $\{G_\alpha\}_{\alpha\in\Lambda}$ of $(\partial A)\cap A$.
Then we can expand this cover to an open cover of $A$ by adding $Int(A)$ to $\{G_\alpha\}_{\alpha\in\Lambda}$.
Now by assumption (using the point $A$ in $X/A$) there is a finite subcollection $G_{\alpha_1}, \dots, G_{\alpha_n}, Int(A)$ such that $\pi G_{\alpha_1}\cup\dots\cup \pi G_{\alpha_n}\cup \pi Int(A)$ is a neighborhood of $A$ in $X/A$.
But $\pi Int(A)\subset \pi G_{\alpha}$ since $G_{\alpha}\cap A\neq \emptyset$ and so $Int(A)$ is not necessary in our finite subcollection. Thus $\pi G_{\alpha_1}\cup\dots\cup \pi G_{\alpha_n}$ is a neighborhood of $A$; let $U$ be 
an open subset of $\pi G_{\alpha_1}\cup\dots\cup \pi G_{\alpha_n}$ containing $A$. Now by looking at the preimages of $U$ and
$\bigcup_{i=1}^n \pi G_{\alpha_i}$ in $X$, we get that
$$A\subset \pi^{-1}(U)\subset \pi^{-1}(\bigcup_{i=1}^n \pi G_{\alpha_i}) = G_{\alpha_1}\cup\dots\cup G_{\alpha_n}\cup A.$$
Since $\pi^{-1}(U)$ is open, then the above expression
implies $A\subset Int(G_{\alpha_1}\cup\dots\cup G_{\alpha_n}\cup A)$. But since all of the $G_{\alpha}$ are open, then $G_{\alpha_1}\cup\dots\cup G_{\alpha_n}\cup A$ is open. Therefore, the first property is necessary.

To see that the second property is necessary, assume that $A$ is not closed and $U$ is any open neighborhood of a fixed $x\in \overline{A}-A$ in $X$. Since $U$ is an open cover of $\pi^{-1}(\pi(x))=x$, then by Theorem \ref{Day and Kelly}, $\pi U$ is a neighborhood of $x$; let $V$ be an open subset of $\pi U$ that contains $x$. 
Then by looking at the preimages of $V$ and $\pi U$, we see (using that $U$ intersects $A$ nontrivially) that
$$A\subset \pi^{-1}(V) \subset \pi^{-1}(\pi U) = U\cup A.$$
But since $\pi^{-1}(V)$ is open, then $A\subset Int(U\cup A)$, i.e. $U\cup A$ is open. Therefore, the second condition is necessary.

Second let's assume the two conditions hold. We will show $\pi$ is a universal quotient map by checking that the conditions of Theorem \ref{Day and Kelly} hold in all three locations in $X/A$ (i.e. for (i) $x = A$, (ii) $x\in X-\overline{A}$, and (iii) $x\in \overline{A}-A$).

(i) For $A\in X/A$, take any open cover $\{G_\alpha\}_{\alpha\in\Lambda}$ of $A$ in $X$. If $A$ is open in $X$, then $\{A\}$ is open in $X/A$ and hence every $\pi G_\alpha$ is a neighborhood. If $A$ is not open, let $\Gamma$ be the finite portion of $\Lambda$ that property 1 guarantees exists, i.e. $A\cup \left(\bigcup_{i\in\Gamma} G_{\alpha_i}\right)$ is open in $X$ and each $G_{\alpha_i}$ intersects $A$ nontrivially. This implies that $\bigcup_{i\in\Gamma} \pi G_{\alpha_i}$ is an open neighborhood of $A$ in $X/A$ (since its preimage is $A\cup \left(\bigcup_{i\in\Gamma} G_{\alpha_i}\right)$). 

(ii) Any $x\in X-\overline{A}$ has an open neighborhood $U_x\subset X-\overline{A}$. Notice that $\pi$ is a homeomorphism on $X-\overline{A}$. Thus for any such $x$ and any open cover $W$ of $\pi^{-1}(x) = x$ in $X$, $\pi W$ is a neighborhood of $x$ because the open neighborhood (in $X/A$) $U_x\cap W$ is contained in $\pi W$. 

(iii) If $A$ is closed, then this is trivial so assume that $A$ is not closed and let $x\in \overline{A}-A$. For any open cover $W$ of $\pi^{-1}(x) = x$ in $X$, $\pi^{-1}(\pi W) = W\cup A$, which is open in $X$ by condition 2. Thus $\pi W$ is an open neighborhood of $x$ in $X/A$. 

Therefore, our two conditions ensure that $\pi$ satisfies Day and Kelly's universal quotient map condition. 
\end{proof}

Corollary \ref{DK corollary} now gives us a way to produce more examples of sieves in the canonical topology:

\begin{ex}
Every quotient of a Hausdorff space by a compact subspace is universal. For example, $\pi\colon D^n\to S^n$ (where $S^n = D^n/\partial D^n$) generates a universal colim sieve.
\end{ex}

\begin{ex}
If $A$ is closed, then $S=\langle \{X\to X/A\}\rangle$ is always a colim sieve. Moreover, it is universal if and only if $\partial A$ is compact. For example, this tells us $\langle\{ \mathbb{R}\to\mathbb{R}/[0,\infty) \}\rangle$ is in the canonical topology and reaffirms that $\langle\{\mathbb{R}\to\mathbb{R}/\mathbb{Z} \}\rangle$ is not.
\end{ex}



%% file: cantopRMod.tex
\section{Universal Colim Sieves in the Category of $R$-modules}\label{UCS in R mod}

The category of $R$-modules does not satisfy the assumptions of Theorem \ref{hinting at basis for can top in special case} or Theorem \ref{basis for can top in special case}.
Indeed, coproducts and pullbacks of $R$-modules do not commute (for example, let $\mathbb{Z}_{(a,b)}$ denote the domain of $\mathbb{Z}\to\mathbb{Z}^2$, $1\mapsto(a,b)$, then we see that $(\mathbb{Z}_{(1,0)}\oplus\mathbb{Z}_{(0,1)})\times_{\mathbb{Z}^2}\mathbb{Z}_{(1,1)}\cong \mathbb{Z}$ but $(\mathbb{Z}_{(1,0)}\times_{\mathbb{Z}^2}\mathbb{Z}_{(1,1)})\oplus (\mathbb{Z}_{(0,1)}\times_{\mathbb{Z}^2}\mathbb{Z}_{(1,1)})\cong 0$).
Thus we do not have basis and presentation results.
Instead, we have some smaller results, reductions and examples.

\begin{note}
Let $R$ be a commutative ring with identity.
We will use \textbf{$R$-Mod} for the category of $R$-modules and \textbf{Ab} for the category of abelian groups.
\end{note}

We start with some basic results.

\begin{cor}\label{sieve with surj is ucs}
Any sieve containing a universal effective epimorphism (e.g. a surjection in \textbf{$R$-Mod} or in \textbf{Sets}) is a universal colim sieve.
\end{cor}

\begin{proof}
This is an immediate consequence of Theorem \ref{ucs is a top} and Corollary \ref{eff epi's gen colim sieves}.
\end{proof}

\begin{lemma}\label{inj map for 2 gen sieves}
In \textbf{$R$-Mod}, if a sieve $S$ on $X$ can be generated by at most two morphisms, then the canonical map $\ds c\colon \colim_{S}{U} \to X$ is an injection.
\end{lemma}

\begin{proof}
Suppose $S = \langle \{f\colon Y\to X, g\colon Z\to X\}\rangle$ and $c(x) = 0$. Since every map in $S$ either factors through $f$ or $g$, then $x$, as an element of $\ds\bigoplus_{A\to X\in S} A$, is really an element $(y,z)\in Y\oplus Z$ in the colimit. So $c(x) = 0$ implies that $y+z=0$ in $X$, i.e. $(y,-z)\in Y\times_X Z$. Thus $y\in Y$ gets identified with $-z\in Z$ in the colimit; hence $(y,z) = (0,z-z) = 0$ in the colimit. Therefore, $x=0$ in the colimit and the map $c$ is an injection.

Using the fact that $\langle \{A_i\to X\}_\alpha\rangle = \langle \{A_i\to X\}_\alpha\cup \{Z\xrightarrow{0}X \}\rangle$, we can say that any sieve generated by one morphism is also generated by two morphsims. This completes the proof.
\end{proof}

\begin{prop}\label{surj gen colim sieves}
In \textbf{$R$-Mod}, let 
$$S = \langle \{f\colon Y\to X\}\rangle \qquad \text{and}\qquad T = \langle \{g\colon U\to X, h\colon V\to X\}\rangle$$ 
be sieves on $X$. Then
\begin{enumerate}
\item $S$ is a universal colim sieve if and only if $f$ is a surjection.
\item $T$ is a colim sieve if and only if $g\oplus h\colon U\oplus V\to X$ is a surjection.
\end{enumerate}
\end{prop}

\begin{proof}
For part 2, Lemma \ref{inj map for 2 gen sieves} tells us that we only need to worry about the surjectivity of $\ds\colim_{T}{U} \to X$ but this is exactly what the above condition is.

For part 1, Lemma \ref{inj map for 2 gen sieves} and Lemma \ref{pb sieve gen set} tell us that we only need worry about the surjectivity of $A\times_X Y \overset{\pi_1}{\longrightarrow} A$ (the generator of $k^\ast S$) for every map $k\colon A\to X$. But $A\times_X Y = \{(a,y)\in A\times Y\, |\, k(a) = f(y) \}$.
Hence $\pi_1$ is a surjection for every map $k$ if and only if $f$ is a surjection.
\end{proof}

\begin{lemma}\label{universality of surjections}
In \textbf{$R$-Mod}, suppose $S = \langle\{ f_i\colon M_i\to R\}_{i\in I}\rangle$ is a sieve on $R$ such that for every $i\in I$ there exists an $a_i\in R$ with $im(f_i) = a_i R$.
If the ideal $(a_i\, |\, i\in I)$ equals $R$, then for every $R$-module homomorphism $g\colon N\to R$, the natural map $\colim_{g^\ast S}{U} \to N$ is a surjection.
\end{lemma}

\begin{proof}
By Proposition \ref{colim is coeq} it suffices to show that 
$\eta\colon\ds\oplus_{i} M_i\times_R N\to N$ is a surjection.
Let $\pi_i\colon M_i\times_R N\to N$ be the natural map.
Fix $x\in N$. Then $a_i g(x)\in a_i R = im(f_i)$ and $a_i g(x)\in im(g)$.
Thus $a_i\cdot x\in im(\pi_i)\subset N$ for all $i\in I$.
Therefore, $x = 1_R\cdot x$ is in $\ds\oplus_{i} im(\pi_i) = im(\eta)$ since R is a unital ring and $(a_i\, |\, i\in I) = R$.
\end{proof}

\begin{prop}\label{good example pf}
Suppose $S = \langle \{ f_1\colon M_1\to R, f_2\colon M_2\to R\} \rangle$ is a sieve on $R$ such that $im(f_i) = a_i R$ for $i = 1,2$.
Then $S$ is in the canonical topology on \textbf{$R$-Mod} if and only if $(a_1,a_2) = R$. 
\end{prop}

\begin{proof}
If $S$ is in the canonical topology, then $S$ is a colim sieve and hence by Proposition \ref{surj gen colim sieves}, $a_1R\oplus a_2R = R$. 

If $(a_1,a_2) = R$, then by Proposition \ref{surj gen colim sieves}, $S$ is a colim sieve. The universality of $S$ follows immediately from Lemma \ref{pb sieve gen set}, Proposition \ref{surj gen colim sieves} and Lemma \ref{universality of surjections}.
\end{proof}

Next we include two results that can help us identify when a sieve is not in the canonical topology.

\begin{prop}\label{easy ring lemma for non ucs}
Let $R$ be any nonzero ring. Let $S = \langle\{ f_i\colon A_i\to X\}_{i\in I}\rangle$ be any sieve on $X$ for any nonzero $R$-module $X$. If there exists a nonzero $b\in X$ such that $span_R(b) \subset (X-\cup_I Im(f_i))\cup \{0\}$, then $S$ is not a universal colim sieve.
\end{prop}

\begin{proof}
Suppose such a $b\in X$ exists. Define $g\colon R\to X$ by $1\to b$. Then $Im(g)\cap Im(f_i) = \{0\}$ for all $i$.
Thus for all $i$, the pullback $R\times_X A_i = ker(g)\times ker(f_i)$ and the image of the natural map $R\times_X A_i\to R$ is $ker(g)$.
In particular, $Im\left(\oplus_i R\times_X A_i \to R\right) = ker(g)$, which by construction is not $R$.
Therefore, $\colim_{g^\ast S}U \to R$ is not surjective and so $g^\ast S$ not a colim sieve on $R$.
\end{proof}

\begin{prop}\label{hope?}
Let $R$ be an infinite principal ideal domain.
Let $$S = \langle \{g_i\colon R^n \hookrightarrow R^n\}_{i=1}^M \cup \{f_i\colon R^{m_i}\hookrightarrow R^n\ |\ m_i<n \}_{i=1}^N \rangle$$ be a sieve on $R^n$. If $S$ is a universal colim sieve, then $g_1\oplus\dots\oplus g_M\colon R^{nM}\to R^n$ is a surjection.
\end{prop}

\begin{proof}
Let $G = g_1\oplus\dots\oplus g_M$. Suppose that $G$ is not a surjection. We will produce a map $\phi$ that shows $S$ is not universal.

By a change of basis (which is allowable by Lemma \ref{cs nice with isom}) we may assume that $G = diag(d_1,d_2,\dots,d_n)$ with $d_i|d_{i+1}$. Because $G$ is not surjectve, then $d_n$ is not a unit.
Indeed, if $d_n$ was a unit, then all of the $d_i$'s would also be units and thus $G$ would be surjective.
By Lemma \ref{nec lin alg} below, there exists an $x\in R^{n-1}$ so that $span_R\{(x,1)\}\cap Im(f_i) = \{0\}$ for all $i=1,\dots,N$. Additionally, since $d_n$ is not a unit, then $(x,1)\not\in Im(G)$.

Define $\phi\colon R\to R^n$ by $1\mapsto (x,1)$. We will show that $\phi^\ast S$ is not a colim sieve. First we will simplify the generating set of $\phi^\ast S$.
By the choice of $x$, the pullback module of $R^{m_i}$ along $\phi$ is $\{0\}$ for all $i=1,\dots,N$.
Therefore, we can write $\phi^\ast S$ as $\phi^\ast S = \langle\{\pi_i\colon R^n\times_{R^n} R \to R\}_{i=1}^M\rangle$ where the $\pi_i$ are the pullbacks of the $g_i$ along $\phi$. Since $(x,1)\not\in Im(G)$ and we have the following commutative diagram
\begin{center}
\begin{tikzcd}
\ds\oplus_{i=1}^M R^n_i\times_{R^n} R
\arrow{r}[above]{\oplus_{i=1}^M\pi_i}
\arrow{d} &
R \arrow{d}[right]{\phi} \\
\ds\oplus_{i=1}^M R^n_i \arrow{r}[below]{G} &
R^n
\end{tikzcd}
\end{center}
%
then $1\not\in Im(\pi_1 \oplus \dots \oplus \pi_M)$. Therefore, $\ds\eta\colon \colim_{\phi^\ast S}{U} \to R$ is not surjective; hence $\phi^\ast S$ is not a colim sieve. 
\end{proof}

Lastly, for completeness we include the linear algebra result referenced in Proposition \ref{hope?}. 

\begin{lemma}\label{nec lin alg}
Let $R$ be an infinite principal ideal domain.
For any finite collection $V_1,\dots,V_N$ of submodules of $R^n$ with $\text{dim}(V_i)<n$, there exists an $x\in R^{n-1}$ such that $span_R\{(x,1)\} \cap V_i = \{0\}$ for all $i$.
\end{lemma}

\begin{proof}
Let $F$ be the quotient field of $R$.
Let
$$W_i = V_i \otimes_R F.$$
%
We will use $F^{n-1}$ to refer to the subspace $\{(a_1,\dots,a_{n-1},0)\ |\ a_i\in F\}$ in $F^n$.
For each $V_i\not\subset F^{n-1}$, fix an element $\nu_i\in V_i$ such that $\nu_i\not\in F^{n-1}$ and write $\nu_i = (v_{i1},\dots,v_{in})$.
Let $\nu_i^0 = (v_{i1},\dots,v_{i(n-1)}, 0)$.
Lastly, for each $V_i\not\subset F^{n-1}$, define a vector space map $\phi_i\colon W_i\to F^{n-1}$ by $w = (w_1,\dots,w_n)\mapsto w - \frac{w_n}{v_{in}}\nu_i$

Ideally, we will find an $x$ such that $(x,1)\not\in W_i$ for all $i$.
So first, let's see what kinds of $(z,1)$ are in $W_i$ by computing $\phi_i(z,1)$.
\begin{align*}
\phi_i(z,1) &= (z,1) - \frac{1}{v_{in}}\nu_i \\
 &= z - \frac{1}{v_{in}}\nu_i^0
\end{align*}
%
Thus
$$z = \phi_i(z,1) + \frac{1}{v_{in}}\nu_i^0.$$
%
Therefore, if $(z,1)\in W_i$, then $z = \phi_i(z,1) + \frac{1}{v_{in}}\nu_i^0$.
Based on this result, define $\Gamma_i = im(\phi_i)\oplus span_F\{\nu_i^0\}$.
So $(z,1)\in W_i$ implies $z\in \Gamma_i$.

For each index $i$ exactly one of the following is true:
\begin{enumerate}
\item $W_i\subset F^{n-1}$,
\item $W_i\not\subset F^{n-1}$ and $\text{dim}_F(\Gamma_i) < n-1$,
\item $W_i\not\subset F^{n-1}$ and $\Gamma_i = F^{n-1}$.
\end{enumerate}

For every index $j$ in collection 1, every $x\in R^{n-1}$ satisfies the equation $span_R\{(x,1)\} \cap V_j = \{0\}$.
Thus when picking our $x$, we only need to consider the indices in collections 2 and 3.

For each index $i$ in collection 2, $\Gamma_i$ is a proper subspace of $F^{n-1}$.
Since there are only finitely many $\Gamma_i$ and $F$ is an infinite field, then there exists a $y = (y_1, \dots, y_{n-1})$ such that $y\neq 0$ and $span_F\{(y,0)\}\cap \Gamma_i = \{0\}$ for all $i$ in collection 2.
By multiplying $y$ by an appropriate $s\in F$ we can clear denominators and so we may assume that $y\in R^{n-1}$.
In particular, for all $r\in R$, $ry\not\in\Gamma_i$, which implies that $(ry,1)\not\in W_i$.
Therefore, for all $r\in R$, $span_R\{(ry,1)\}\cap V_i = \{0\}$ for all indices in collection 2.

Continuing with the $y$ from the previous paragraph, we now consider the indices $k$ in collection 3 and their corresponding $\Gamma_k$.
In this situation, $(y,0)\in \Gamma_k$, i.e. $y = \phi_k(z) + u_k\nu_k^0$ for some $z\in W_k$ and $u_k\in F$.
Since $R$ is an infinite ring and collection 3 contains finitely many indices $k$, we can pick a nonzero $\rho\in R$ such that for all $k$, $\rho u_k\in R$ and $\rho u_k \neq \frac{1}{v_{kn}}$.
Thus $\rho y \neq \phi_k(a) + \frac{1}{v_{kn}}\nu_k^0$ for any $a\in W_k$, which implies that $(\rho y,1)\not\in W_k$.
Therefore, $span_R\{(\rho y,1)\}\cap V_k = \{0\}$ for all indices in collection 3.

We can take $x = \rho y$.

\end{proof}

\noindent\textsc{Examples}
\bigskip

Here we include a few examples and non-examples of sieves in the canonical topology for various rings $R$.

\begin{ex}
In the category of $R$-modules 
every surjective map generates a universal colim sieve (see
Proposition \ref{surj gen colim sieves}). As more specific examples, the sieve $\langle\{ \mathbb{Z}\overset{\pi}{\longrightarrow} \mathbb{Z}/n\mathbb{Z}\ |\ 1\mapsto 1\}\rangle$ is in the canonical topology on \textbf{Ab} and in \textbf{$R$-Mod}, the sieve $\langle\{R^n\to R\ |\ (a_1,\dots,a_n)\mapsto a_1\}\rangle$ is in the canonical topology. 
\end{ex}

\begin{ex}
By Proposition \ref{good example pf}, $\langle\{R\overset{a}{\longrightarrow} R, R\overset{b}{\longrightarrow} R\}\rangle$ is in the canonical topology if and only if $(a,b) = R$. As more specific examples, in \textbf{Ab} the sieve $\langle\{\mathbb{Z}\overset{2}{\longrightarrow} \mathbb{Z}, \mathbb{Z}\overset{3}{\longrightarrow} \mathbb{Z}\}\rangle$ is in the canonical topology; and when the function $\cdot g(x)\colon C^\infty(\mathbb{R})\to C^\infty(\mathbb{R})$ is the map $f(x)\mapsto (g\cdot f)(x)$, then the sieve $\langle\{C^\infty(\mathbb{R})\overset{\cdot x}{\longrightarrow} C^\infty(\mathbb{R}), C^\infty(\mathbb{R})\overset{\cdot \sin(x)}{\longrightarrow} C^\infty(\mathbb{R})\}\rangle$ is not in the canonical topology on $C^\infty(\mathbb{R})$-modules.
\end{ex}

\begin{ex}\label{good counterex}
The sieve $S = \langle\{R\overset{i_1}{\to} R^2, R\overset{i_2}{\to} R^2 \}\rangle$ where $i_1(1) = (1,0)$ and $i_2(1) = (0,1)$ 
(in the category of $R$-modules for nontrivial $R$) is not in the canonical topology.
By Proposition \ref{surj gen colim sieves}, $S$ is clearly a colim sieve so to see that $S$ is not universal consider the map $\Delta\colon R\to R^2$, $1\mapsto (1,1)$.
Then for $k=1,2$, $i_k$ pulled back along $\Delta$ yields the zero map $z\colon 0\to R$. 
Hence Lemma \ref{pb sieve gen set} says $\Delta^\ast S = \langle\{z\colon 0\to R\}\rangle$, which is clearly not a colim sieve. 

Similarly  $\langle\{R\overset{i_k}{\to} R^n\ |\ k = 1,\dots,n\}\rangle$ is a colim sieve but is not in the canonical topology. (This is also a consequence of Proposition \ref{easy ring lemma for non ucs}.)
\end{ex}

\begin{ex}
Let $S = \langle\{ f_k\colon \mathbb{Q}\to\mathbb{Q}[t]\ |\ f_k(1) = 1+t+\dots+t^k \}_{k=1}^\infty\rangle$ in the category of rational vector spaces. This $S$ is not in the canonical topology.
(This is a direct consequence of Proposition \ref{easy ring lemma for non ucs} using $b = t$.)
\end{ex}

\begin{ex}
Let $F$ be an infinite field. In the category of $F$ vector spaces, a sieve of the form $S = \langle\{ F^{m_i}\hookrightarrow F^n\ |\ m_i\leq n \}_{i=1}^M\rangle$ is in the canonical topology if and only if $m_i = n$ for some $i$ if and only if $S$ contains an isomorphism. (This is a consequence of 
Proposition \ref{hope?}.)
\end{ex}

\begin{prop}
Consider the diagram $B_1\hookrightarrow B_2\hookrightarrow B_3\hookrightarrow \dots$ made with only injective maps and the direct limit $B\coloneqq \colim B_n$ in \textbf{$R$-mod}. Let the maps $\iota_n\colon B_n\to B$ be the natural maps into the colimit.
Then the sieve $\langle\{\iota_n \,|\, n\in\mathbb{N}\}\rangle$ is a universal colim sieve.
\end{prop}

\begin{proof}
Let $\Gamma\colon\mathbb{N}\to S$ by $n\mapsto\iota_n$.
Notice that $\Gamma$ is a final functor; this is easy to see since the injectivity of $\iota_n$ and the maps in our diagram imply that $B_i\times_B B_j\cong B_{min(i,j)}$.
Thus $\colim_{S}U$ exists and $\colim_{S}U\cong\colim_{\mathbb{N}}U\Gamma \cong B$.
Therefore, $S$ is a colim sieve.

To see that $S$ is universal, let $f\colon X\to B$ and set $X_i \coloneqq X\times_B B_i$.
For each $n\in\mathbb{N}$, $\iota_n$ and $B_n\to B_{n+1}$ are both injective maps; this implies that the natural maps $X_n\to X_{n+1}$ and $X_n\to X$ are also injective maps since the pullback of an injection in \textbf{$R$-Mod} is an injection and $X_i\cong X_{i+1}\times_{B_{i+1}}B_i$.
Additionally, it is an easy exercise to see that the direct limit $\colim X_i$ is isomorphic to $X$.
In other words, $f^\ast S$ is the type of sieve described in the assumptions of this proposition and proved to be a colim sieve in the previous paragraph.
\end{proof}

\begin{ex}
Take $B_n = \mathbb{R}^n$ and let $B_n\to B_{n+1}$ be the inclusion map $(x_1,\dots,x_n)\mapsto(x_1,\dots,x_n,0)$.
Use $\mathbb{R}^\infty$ to denote the direct limit.
Then the above proposition shows that $\langle\{\mathbb{R}^n\hookrightarrow\mathbb{R}^\infty\}_{n\in\mathbb{N}} \rangle$ is in the canonical topology on the category of $\mathbb{R}$ vector spaces.
(Compare this to Example \ref{direct limit top example}.)
\end{ex}

\bigskip

\noindent\textsc{Reductions}
\bigskip

In this part we prove some reductions that allow us to limit our view (of sieve generating sets and the maps universality must be checked over) to the non-full subcategory of free modules with injective maps when $R$ is `nice.'
The first reduction will be reducing the types of sieves we need to look at:

\begin{prop}[Reduction 1]\label{reduction 1}
In \textbf{$R$-Mod}, let $S$ be a sieve on $X$.
Then the following are equivalent
\begin{enumerate}
\item $S$ is a universal colim sieve
\item $f^\ast S$ is a universal colim sieve for every surjection $f\colon Y\to X$
\item $f^\ast S$ is a universal colim sieve for some surjection $f\colon Y\to X$
\end{enumerate}
\end{prop}

\begin{proof}
It is obvious that 1 implies 2 and 2 implies 3, so
it suffices to show 3 implies 1.

Assume $f^\ast S$ is a universal colim sieve for some fixed surjection $f\colon Y\to X$. Set $T = \langle \{f\colon Y\to X\}\rangle$. By Proposition \ref{surj gen colim sieves}, $T$ is a universal colim sieve since $f$ is a surjection.
We will now use $T$ together with the Grothendieck topology's transitivity axiom to show that $S$ is a universal colim sieve. Notice that $S$ satisfies the hypotheses of this axiom with respect to $T$.
Indeed, since every $g\in T$ factors as $f\circ k$ for some $k$, 
then $g^\ast S = (fk)^\ast S = k^\ast(f^\ast S)$, which implies that $g^\ast S$ is a universal colim sieve (as $f^\ast S$ is universal) for every $g\in T$. Therefore, by the transitivity axiom of a Grothendieck topology, $S$ is a universal colim sieve. 
\end{proof}

To rephrase our first reduction: $S$ is a universal colim sieve on $X$ if and only if $f^\ast S$ is a universal colim on $R^n$ where $f\colon R^n\to X$ is a surjection (note that $n$ is not necessarily assumed to be finite).
This reduction means that we can restrict our view to free modules (not necessarily finitely generated).
Specifically, we only need to look at sieves on free modules and check the universality condition on free modules.
Indeed, $S$ is a universal colim sieve on $X$ if and only if for all $g\colon Y\to X$, $g^\ast S$ is a universal colim sieve on $Y$ if and only if for all $g\colon Y\to X$, $(gf)^\ast S$ is a universal colim sieve on $R^n$ for some surjection $f\colon R^n\to Y$.

\begin{prop}[Reduction 2]\label{reduction 2}
In $R$\textbf{-Mod} when $R$ is a principal ideal domain, every sieve on $R^n$ equals a sieve of the form
\begin{equation*}
\langle \{g_i\colon R^{m_i}\hookrightarrow R^n\colon m_i\leq n\}_{i\in I}\rangle
\end{equation*}
%
where the $g_i$ are injections.
\end{prop}

\begin{proof}
Let $S = \langle\{f_i\colon A_i\to R^n\}_{i\in I}\rangle$ be a sieve on $R^n$. Set $$T = \langle\{g_i\colon Im(f_i)\to R^n\}_{i\in I}\rangle$$ where the $g_i$'s are inclusion maps.
Since $R$ is a PID and $Im(f_i)$ is a submodule of $R^n$, then $Im(f_i)\cong R^{m_i}$ for some $m_i\leq n$.
Thus $T$ is of the desired form and we will show that $S=T$. First notice that $S\subset T$. To get that $T$ is a subcollection of $S$, notice that $\tilde{f}_i\colon A_i\to Im(f_i)$ (i.e. $f_i$ with a different codomain) is split because $\tilde{f}_i$ is a surjective map onto a projective module; call the splitting $\chi_i$. Hence $g_i = g_i\circ \tilde{f}_i\circ\chi_i = f_i\circ\chi_i$ implies that $T\subset S$ and completes the proof.
\end{proof}

To rephrase our second reduction: when talking about sieves on $R^n$, we only need to talk about sieves generated by injections of free modules. Thus we can restrict our view of sieve generating sets to the non-full subcategory of free modules with injective morphisms.

Our next reduction will also assume $R$ is a principal ideal domain. In particular, fix $n$ and a map $f\colon X\to R^n$ for some $R$-module $X$. Then
since $R$ is a PID, we may write
$$
X\cong R^m\oplus K \quad \text{for some $m\leq n$, where}
$$$$
R^m\cong Im(f), \quad
K = ker(f), \quad
f = g+z \quad \text{with}
$$$$
g\colon R^m\to R^n \text{ an injection and } z\colon K\to R^n \text{ the zero map}.
$$

\begin{prop}[Reduction 3]\label{reduction 3}
Let $R$ be a principal ideal domain, $S$ be a sieve on $R^n$ in \textbf{$R$-Mod} and $f\colon X\to R^n$.
Then, using the set-up described in the previous paragraph,
$$\colim_{f^\ast S}{U} \cong \left(\colim_{g^\ast S}{U}\right) \oplus \left(\colim_{z^\ast S}{U}\right).$$
Moreover, $z^\ast S$ is a universal colim sieve; hence $f^\ast S$ is a colim sieve if and only if $g^\ast S$ is a colim sieve.
\end{prop}

\begin{proof}[Sketch of Proof.]
By Proposition \ref{reduction 2}, we may assume that $S$ can be written in the form $S = \langle \{\eta_i\colon R^{p_i}\hookrightarrow R^n\colon p_i\leq n\}_{i\in I}\rangle$.
Consider the diagrams $\EuScript{X}$, $\EuScript{R}$ and $\EuScript{K}$ defined as:
\begin{center}
$\EuScript{X} = \left(\begin{tikzcd}
\bigoplus_{i\in I}(R^{p_i} \times_{R^n} X)\times_X (R^{p_i} \times_{R^n} X)
\arrow[d, shift left = 2] \arrow[d, shift right = 2] \\
\bigoplus_{i\in I}(R^{p_i} \times_{R^n} X)
\end{tikzcd}\right)$, \\
$\EuScript{R} = \left(\begin{tikzcd}
\bigoplus_{i\in I}(R^{p_i} \times_{R^n} R^m)\times_{R^m} (R^{p_i} \times_{R^n} R^m)
\arrow[d, shift left = 2] \arrow[d, shift right = 2] \\
\bigoplus_{i\in I}(R^{p_i} \times_{R^n} R^m)
\end{tikzcd}\right), \text{ and }$ \\
$\EuScript{K} = \left(\begin{tikzcd}
\bigoplus_{i\in I}(R^{p_i} \times_{R^n} K)\times_K (R^{p_i} \times_{R^n} K)
\arrow[d, shift left = 2] \arrow[d, shift right = 2]\\
\bigoplus_{i\in I}(R^{p_i} \times_{R^n} K)
\end{tikzcd}\right)$
\end{center}

First we look at the objects of $\EuScript{X}$.
Since each $\eta_i$ is injective, then for all $i$
$$R^{p_i} \times_{R^n} X \cong (R^{p_i} \times_{R^n} R^m)\oplus (R^{p_i} \times_{R^n} K)$$
%
and for all $i$, $q$
\begin{align*}
(R^{p_i} &\times_{R^n} X) \times_X (R^{p_q}\times_{R^n} X) \\
   &\cong ((R^{p_i}\times_{R^n} R^m)\times_{R^m}(R^{p_q}\times_{R^n} R^m))\oplus ((R^{p_i}\times_{R^n} K)\times_K(R^{p_q}\times_{R^n} K)).
\end{align*}
%
In other words, $\EuScript{X} \cong \EuScript{R}\oplus \EuScript{K}$.
But since colimits ``commute'' with colimits, then $\Coeq(\EuScript{X}) \cong \Coeq(\EuScript{R})\oplus \Coeq(\EuScript{K})$. Now by Lemma \ref{pb sieve gen set} and Proposition \ref{colim is coeq}, the first part has been proven, i.e.
$$\colim_{f^\ast S}{U} \cong \left(\colim_{g^\ast S}{U}\right) \oplus \left(\colim_{z^\ast S}{U}\right).$$

Next we notice that $z^\ast S$ is a universal colim sieve.
Indeed, since $\eta_i$ is an injection and $z$ is the zero map, it easily follows that $z^\ast S = \langle \{id\colon K\to K\}\rangle$.

To complete the proof, notice that we have the following commutative diagram
\begin{center}
\begin{tikzcd}
\Coeq(\EuScript{X}) \cong \Coeq(\EuScript{R}) \oplus \Coeq(\EuScript{K})
\arrow[d, shift right = 1, "\rho"] \arrow[d, bend right = 25, shift right = 14, "\chi"] \arrow[d, bend left = 25, shift left = 14, "\kappa"] \\
X \cong R^m \oplus K
\end{tikzcd}
\end{center}
%
where the vertical maps are the obvious canonical maps.
This $\chi = \rho\oplus\kappa$ is an isomorphism if and only if both $\rho$ and $\kappa$ are isomorphisms.
We have already shown that $\kappa$ is an isomorphism (as $z^\ast S$ is a universal colim sieve), thus this diagram implies that $\chi$ is an isomorphism if and only if $\rho$ is; hence $f^\ast S$ is colim sieve if and only if $g^\ast S$ is a colim sieve.

\end{proof}

Lastly, we rephrase our third reduction:

\begin{cor}
When $R$ is a PID, a sieve on $R^n$ is a universal colim sieve if and only if $f^\ast S$ is a colim sieve for every injection $f\colon R^m\to R^n$.
\end{cor}

All together our reductions basically allow us to work in the subcategory of free modules with injective morphisms instead of in $R$\textbf{-Mod}.

\subsection{The Category of Abelian Groups}\label{cat of abgr}

This section will be primarily made up of examples. Additionally, we include a characterization of sieves on $\mathbb{Z}$ and one result for sieves on larger free abelian groups.

\begin{ex}\label{rel prime ex}
By Corollary \ref{good example pf}, $\langle\{ \mathbb{Z} \xrightarrow{\times a} \mathbb{Z}, \mathbb{Z}\xrightarrow{\times b} \mathbb{Z} \}\rangle$ is a universal colim sieve if and only if $a$ and $b$ are relatively prime. 
\end{ex}

\begin{ex}\label{sieve ex}
The sieve $S = \langle \{\mathbb{Z}\xrightarrow{\times 1} \mathbb{Z}/4\mathbb{Z}, \mathbb{Z}/2\mathbb{Z}\xrightarrow{\times 2}\mathbb{Z}/4\mathbb{Z}\}\rangle$ is a universal colim sieve on $\mathbb{Z}/4\mathbb{Z}$ by Corollary \ref{sieve with surj is ucs}. 
Additionally, $S$ is not monogenic, i.e. it cannot be written as a sieve generated by one morphism.

\end{ex}


\begin{ex}
Let $S = \langle \{g\colon \mathbb{Z}^n \hookrightarrow \mathbb{Z}^n\} \cup \{f_i\colon \mathbb{Z}^{m_i}\hookrightarrow \mathbb{Z}^n\ |\ m_i<n \}_{i=1}^N \rangle$ be a sieve on $\mathbb{Z}^n$. Then $S$ is a universal colim sieve if and only if $g$ is a surjection, i.e. $g$ is an isomorphism. (This is a direct corollary of Proposition \ref{hope?} and Corollary \ref{sieve with surj is ucs}.)
\end{ex}



Ideally, we would like to know a `nice' basis for the canonical topology on \textbf{Ab}, like the bases in Section \ref{section basis for Top and Sets}; to start moving towards this ideal, we look
at the simplest free group, $\mathbb{Z}$. In Example \ref{rel prime ex} we see that a relative prime pair of numbers will generate a universal colim sieve; this is actually true in general, specifically:

\begin{prop}\label{sieves on Z}
Let $S = \langle \{ \mathbb{Z}\xrightarrow{\times a_i} \mathbb{Z} \}_{i=1}^N \rangle$ be a sieve on $\mathbb{Z}$. Then $S$ is a universal colim sieve if and only if $\text{gcd}(a_1,\dots,a_N) = 1$.
\end{prop}

\begin{proof}
First assume that $S$ is a universal colim sieve. In particular, the map $\colim_{S}{U}\to\mathbb{Z}$ is a surjection, i.e. $\mathbb{Z}^N\to\mathbb{Z}$, $(x_1,\dots,x_N) \mapsto a_1x_1 + \dots + a_Nx_N$ is a surjection. Therefore, $(a_1,\dots,a_N) = \mathbb{Z}$ and this proves the forward direction.

Now assume that $\text{gcd}(a_1,\dots,a_N) = 1$.
We will break the proof that $S$ is a universal colim sieve up into several pieces.
First we will reduce the proof to showing that $S$ is a colim sieve. By the reductions (Propositions \ref{reduction 1}, \ref{reduction 2} and \ref{reduction 3}), universality only needs to be checked along maps of the form $f\colon \mathbb{Z}\xrightarrow{\times k}\mathbb{Z}$ where $k\neq 0$. Fix $k\neq 0$, i.e. fix $f$,
and write $\mathbb{Z}_b$ for the domain of $\mathbb{Z}\xrightarrow{\times b}\mathbb{Z}$.
By Lemma \ref{pb sieve gen set}, $f^\ast S = \langle \{ \pi_i\colon \mathbb{Z}_{a_i}\times_\mathbb{Z} \mathbb{Z}_k \to \mathbb{Z}_k \}_{i=1}^N \rangle$.
Moreover, it is easy to see that the pullback $\mathbb{Z}_{a_i}\times_\mathbb{Z} \mathbb{Z}_k\cong \mathbb{Z}$ and $\pi_i$ must be multiplication by $\frac{a_i}{\text{gcd}(a_i,k)}$.
Since $\text{gcd}(a_1,\dots,a_N)$ equals $1$, then $\text{gcd}\left(\frac{a_1}{\text{gcd}(a_1,k)},\dots,\frac{a_N}{\text{gcd}(a_N,k)}\right) = 1$ and hence $f^\ast S$ has the same form as $S$.
Specifically, any argument showing that $S$ is a colim sieve will similarly show that $f^\ast S$ is a colim sieve. Therefore, it suffices to show that $S$ is a colim sieve.

To see that $S$ is a colim sieve, i.e. to see that the map $\colim_{S}{U}\to \mathbb{Z}$ induced by $a_1,\dots,a_N$ is an isomorphism, let $\alpha = \frac{N(N-1)}{2}$ 
and notice that
\begin{equation*}
\begin{split}
\colim_{S}{U}
& \cong \Coeq\left( \begin{tikzcd}
\oplus_{i=1}^\alpha \mathbb{Z}
\arrow[d, shift left = 2] \arrow[d, shift right = 2] \\
\oplus_{i=1}^{N} \mathbb{Z}
\end{tikzcd}\right) \\
& \cong \text{Cokernel}\left(\phi\colon \mathbb{Z}^\alpha \to \mathbb{Z}^N\right)
\end{split}
\end{equation*}
%
for some map $\phi$ where the first isomorphism comes from Lemma \ref{colim is coeq} and the last isomorphism comes from the fact that we are working in an abelian category.
Now this map $\phi$ happens to be the third map in the Taylor resolution of $\mathbb{Z}$, i.e. $\phi_1$ in \cite{taylorres}.
We make two remarks about this previous sentence: (1) we will not prove that our $\phi$ is \cite{taylorres}'s $\phi_1$, although this is easy to observe, and (2) the Taylor resolution in \cite{taylorres} is specifically for polynomial rings, not $\mathbb{Z}$, however, both the definition of the Taylor resolution and the proof that it is in fact a free resolution are analogous.
Here is the end of the Taylor resolution:
$$\dots \to \mathbb{Z}^\alpha \xrightarrow{\phi} \mathbb{Z}^N \xrightarrow{(a_1\ \dots\ a_N)} \mathbb{Z}\to \mathbb{Z}/(a_1,\dots,a_N)\mathbb{Z} \to 0$$
%
Since $\text{gcd}(a_1,\dots,a_N) = 1$, then it follows that $(a_1\ \dots\ a_N)$ is a surjection and $\mathbb{Z}/(a_1,\dots,a_N)\mathbb{Z} \cong 0$. Thus we obtain $0\to Im(\phi) \to \mathbb{Z}^N \to \mathbb{Z}\to 0$, which is an exact sequence and hence implies that the cokernel of $\phi$ is $\mathbb{Z}$. Additionally, since $(a_1\ \dots\ a_N)$ induced our map $\colim_{S}{U}\to \mathbb{Z}$, then this short exact sequence also says that $S$ is a colim sieve.
\end{proof}

Because of Proposition \ref{sieves on Z}, we can now easily determine when a sieve on $\mathbb{Z}$ is in the canonical topology and we can easily come up with examples; for example,
$\langle \{ \mathbb{Z}\xrightarrow{\times 15}\mathbb{Z}, \mathbb{Z} \xrightarrow{\times 10} \mathbb{Z}, \mathbb{Z} \xrightarrow{\times 12} \mathbb{Z} \} \rangle$ is in the canonical topology whereas the sieve
$\langle \{ \mathbb{Z} \xrightarrow{\times 15} \mathbb{Z}, \mathbb{Z} \xrightarrow{\times 50}\mathbb{Z}, \mathbb{Z} \xrightarrow{\times 20} \mathbb{Z} \}\rangle$ is not.
One may hope for a similar outcome for sieves on $\mathbb{Z}^n$ when $n\geq 2$, however, the Taylor resolution used in the proof of Proposition \ref{sieves on Z}
does not seem to generalize in a suitable manner.
Instead, we have a proposition that may tell us when a potential sieve is not in the canonical topology.

\begin{prop}\label{diag matrices}
Let $S = \langle \{ \mathbb{Z}^n \xrightarrow{A_i} \mathbb{Z}^n \}_{i=1}^N \rangle$ where $A_i$ is a diagonal matrix with $\det(A_i)\neq 0$. Then there exists a map $\beta\colon \mathbb{Z}\to \mathbb{Z}^n$ such that $\beta^\ast S$ is not a colim sieve if and only if $\text{gcd}(\det(A_1),\dots,\det(A_N)) \neq 1$.
\end{prop}

\begin{proof}
First we set up some notation:
Let $A_i = diag(a_{1i},\dots,a_{ni})$ and $\mathbb{Z}^n_i$ be the domain of $A_i$.

To prove the backward direction, suppose that $\text{gcd}(\det(A_1),\dots,\det(A_N))$ does not equal $1$. We can rephrase the assumptions as $a_{ik}\neq 0$ for all $k$ and there exists a prime $q$ such that $q$ divides the product $a_{1i}\dots a_{ni}$ for all $i$. Set $\beta$ equal to the diagonal embedding, i.e. $1\mapsto (1,\dots,1)$.
Then by Lemma \ref{pb sieve gen set}, $\beta^\ast S = \langle \{ f_i\colon \mathbb{Z}^n_i \times_{\mathbb{Z}^n} \mathbb{Z} \to \mathbb{Z}\}_{i=1}^N \rangle$.
Let $k_i = \text{lcm}(a_{1i},\dots, a_{ni})$ and $\chi_i\colon \mathbb{Z}\to \mathbb{Z}^n$, $1\mapsto \left( \frac{k_i}{a_{1i}}, \dots, \frac{k_i}{a_{ni}}  \right)$, then
\begin{center}
\begin{tikzcd}
\mathbb{Z} \arrow{d}[left]{k_i} \arrow{r}[above]{\chi_i} &
\mathbb{Z}^n \arrow{d}[right]{A_i} \\
\mathbb{Z} \arrow{r}[below]{\beta} & \mathbb{Z}^n
\end{tikzcd}
\end{center}
%
is a pullback diagram.
Moreover, the prime $q$ divides $k_i$ for all $i$ since it divides $a_{1i}\dots a_{ni}$ for all $i$.
Thus $\text{gcd}(k_1,\dots,k_N)\neq 1$.
Now by Proposition \ref{sieves on Z}, we can see that $\beta^\ast S = \langle \{\mathbb{Z} \xrightarrow{\times k_i} \mathbb{Z}\}_{i=1}^N \rangle$ is not a universal colim sieve. 
In particular, the first part of the proof of Proposition \ref{sieves on Z} shows that $\beta^\ast S$ is not a colim sieve.

To prove the forward direction, we will prove the contrapositive statement.
So suppose that $\text{gcd}(\det(A_1),\dots,\det(A_N)) = 1$.
Let $\beta\colon \mathbb{Z}\to \mathbb{Z}^n$ be given as the matrix $\begin{pmatrix} b_1 \\ \vdots \\ b_n \end{pmatrix}$.
To see that $\beta^\ast S = \langle \{ f_i\colon \mathbb{Z}^n_i \times_{\mathbb{Z}^n} \mathbb{Z} \to \mathbb{Z}\}_{i=1}^N \rangle$ is a colim sieve,
notice that we have the pullback diagram
\begin{center}
\begin{tikzcd}
\mathbb{Z} \arrow{d}[left]{k_i} \arrow{r} &
\mathbb{Z}^n \arrow{d}[right]{A_i} \\
\mathbb{Z} \arrow{r}[below]{\beta} & \mathbb{Z}^n
\end{tikzcd}
\end{center}
%
where $k_i = \text{lcm} \left(\frac{a_{1i}}{\text{gcd}(a_{1i},b_1)},\dots,\frac{a_{ni}}{\text{gcd}(a_{ni},b_n)}\right)$.
Hence, $k_i$ divides $\det(A_i)$. 
This implies that $\text{gcd}(k_1,\dots,k_n)$ divides $\text{gcd}(\det(A_1),\dots,\det(A_N))$ and hence equals 1.
Now by Proposition \ref{sieves on Z}, we can see that $\beta^\ast S = \langle \{\mathbb{Z} \xrightarrow{\times k_i} \mathbb{Z}\}_{i=1}^N \rangle$ is a universal colim sieve. 
\end{proof}

\begin{ex}
Based on Proposition \ref{diag matrices} we can automatically say that the sieve $\displaystyle \left\langle \left\{
\begin{pmatrix}
4 & 0 \\
0 & 14
\end{pmatrix},
\begin{pmatrix}
21 & 0 \\
0 & 2
\end{pmatrix},
\begin{pmatrix}
1 & 0 \\
0 & 49
\end{pmatrix}
\right\} \right\rangle$ on $\mathbb{Z}^2$ is not in the canonical topology because each matrix has a multiple of 7 somewhere on its diagonal.
\end{ex}

Suppose, like in Proposition \ref{diag matrices}, $S = \langle \{ \mathbb{Z}^n \xrightarrow{A_i} \mathbb{Z}^n \}_{i=1}^N \rangle$ where each $A_i$ is a diagonal matrix and $\text{gcd}(\det(A_1),\dots,\det(A_N)) = 1$.
In order to determine if $S$ is a universal colim sieve, we (only)
need to check if $f^\ast S$ is a colim sieve for all $f\colon \mathbb{Z}^m \hookrightarrow \mathbb{Z}^n$, $2\leq m\leq n$. 
However, this is still a fair amount of work 
and it would be nice if this process could be simplified further.

Now we finish this section with a few more examples.
Note: we will not prove any assertions in these examples, however, they are all basic computations that can be checked using undergraduate linear algebra.

\begin{ex}
The sieve $S_1 = \displaystyle \left\langle \left\{
\begin{pmatrix}
7 & 0 \\
1 & 4
\end{pmatrix},
\begin{pmatrix}
21 & 0 \\
1 & 18
\end{pmatrix},
\begin{pmatrix}
24 & 0 \\
6 & 5
\end{pmatrix}
\right\} \right\rangle$ on $\mathbb{Z}^2$ is not in the canonical topology although it is a colim sieve. In particular, $S_1$ is not universal because $f^\ast S_1$ is not a colim sieve for $f\colon \mathbb{Z} \to \mathbb{Z}^2$, $f(1) = (1,0)$.
\end{ex}

If we take the generating set of $S_1$ and change the 1 in the first matrix to a 0, then we get the following example:

\begin{ex}
The sieve $S_2 = \displaystyle \left\langle \left\{
\begin{pmatrix}
7 & 0 \\
0 & 4
\end{pmatrix},
\begin{pmatrix}
21 & 0 \\
1 & 18
\end{pmatrix},
\begin{pmatrix}
24 & 0 \\
6 & 5
\end{pmatrix}
\right\} \right\rangle$ on $\mathbb{Z}^2$ is not a colim sieve since $\colim_{S}{U}\cong \mathbb{Z}^2\oplus \mathbb{Z}/2\mathbb{Z}$. Therefore, $S_2$ is also not in the canonical topology.
\end{ex}

Finally, if take the generating set of $S_2$ and change the 18 in the second matrix to a 9, then we get:

\begin{ex}
The sieve $S_3 = \displaystyle \left\langle \left\{
\begin{pmatrix}
7 & 0 \\
0 & 4
\end{pmatrix},
\begin{pmatrix}
21 & 0 \\
1 & 9
\end{pmatrix},
\begin{pmatrix}
24 & 0 \\
6 & 5
\end{pmatrix}
\right\} \right\rangle$ on $\mathbb{Z}^2$ is a colim sieve, however, whether or not this sieve is in the canonical topology is unknown.
\end{ex}
